\documentclass{article}
\usepackage[utf8]{inputenc}

\usepackage{biblatex}
\usepackage[margin=1in]{geometry}
\usepackage{amssymb}
\usepackage{amsmath} 
\usepackage{amsthm}
\usepackage{xcolor}
\usepackage{hyperref}
\usepackage{comment}
\usepackage{caption}

\usepackage{tikz}
\usepackage{pgfplots}
\pgfplotsset{compat=1.15}
\usepackage{mathrsfs}
\usetikzlibrary{arrows}

\let\div\undefined
\DeclareMathOperator{\div}{div}

\DeclareMathOperator{\dist}{dist}
\DeclareMathOperator{\Top}{Top}

\allowdisplaybreaks

\newtheorem{thm}{Theorem}

\newtheorem{lemma}[thm]{Lemma}

\newtheorem{prop}[thm]{Proposition}
\newtheorem{claim}[thm]{Claim}

\title{A note on the maximization of the first Dirichlet eigenvalue for perforated planar domains}
\author{Manuel Dias }
\date{December 2022}

\begin{document}

\maketitle

\begin{abstract}
In this work we prove that given an open bounded set $\Omega \subset \mathbb{R}^2$ with a $C^2$ boundary, there exists $\epsilon := \epsilon(\Omega)$ small enough such that for all $0 < \delta < \epsilon$ the maximum of $
    \{
        \lambda_1(\Omega - B_{\delta}(x))
        :
        B_{\delta} \subset \Omega
    \}
$ is never attained when the ball is close enough to the boundary. In particular it is not obtained when $B_\delta(x)$ is touching the boundary $\partial \Omega$. 
\end{abstract}

\section{Introduction}
Consider the Dirichlet eigenvalue problem, in a bounded domain $\Omega \subset \mathbb{R}^2$, given by:
\begin{equation}
\begin{cases}
  -\Delta u = \lambda u  & \text{ in } \Omega \\
   \quad u = 0 &  \text{ on } \partial \Omega
\end{cases}
\end{equation}
where $-\Delta$ admits a purely discrete spectrum $0<\lambda_1(\Omega)< \lambda_2(\Omega)\leq ... \leq \lambda_i(\Omega)\leq ... \rightarrow \infty$. Observe that $\lambda_1(\Omega)$ is a simple eigenvalue, and as a consequence there exists a unique associated positive normalized eigenfunction $u_\Omega$.

In this paper, we study the maximization problem

\begin{equation}
\label{problemSet}
    \max_{B \in \mathcal{B}_\delta} \lambda_1(\Omega - B),
\end{equation}
where $\mathcal{B}_\delta$ is the set of balls $B_\delta(x)$ of fixed radius $\delta>0$ such that $B_\delta(x) \subset \Omega$. It follows from the work of Flucher \cite{approxEigenHole} that, for every point $x \in \Omega$, there exists some radius $R_x$ such that $B_{R_x}(x) \subset \Omega$, and:
\begin{equation}
\label{asymptoticFormula}
    \lambda_1(\Omega - B_r(x))
=
    \lambda_1(\Omega)
+
    u_{\Omega}^2(x)
    \frac{2\pi}{-\log(r/R_x)}
+
    o(\log^{-1}(r)) \quad \text{as } r\downarrow 0.
\end{equation}
This result hints that the maximum of \eqref{problemSet} is attained at a ball $B_\delta(x)$ where $x$ is the maximum of $u_{\Omega}$. However, since $R_x$ in \eqref{asymptoticFormula} depends on the choice of $x$, we can't compare uniformly $\lambda_1(\Omega - B_r(x))$ using this asymptotic formula. In particular it doesn't allow us to consider small balls that are close to the boundary, since $B_r(x)$ must be contained in the bigger ball $B_{R_x}(x) \subset \Omega$.

The first results regarding the maximum of \eqref{problemSet} were obtained when $\Omega$ is a ball $B_R(0)$. In \cite{Anel2D}, Ramm and Shivakumar proved that the maximum of \eqref{problemSet} is attained when $B_\delta(x)$ is concentric with $B_R(0)$. The result was generalized to higher dimensions by Kesavan in \cite{ANELND}. Regarding higher eigenvalues, in \cite{2ndEigenvalueAnel}, El Soufi and Kiwan proved that the maximum of $\lambda_2(B_R(0)-B_\delta(x))$ is attained when the balls are concentric. In the case where $\Omega$ is a convex set, or presents some symmetry, Harrell, Kurata and Kr\"oger proved in \cite{convexDrumsHole} that the maximizer is not attained when the ball touches the boundary. Still in the case where $\Omega$ is a ball $B_R(0)$, in another recent paper \cite{DihedralStuff} by Chorwadwala and Roy, the authors showed that the maximizer $K$ of $\lambda_1(B_R(0) - K)$, for sets $K$ with dihedral symmetry and $K \subset B_R(0)$, must share some symmetry with $B_R(0)$.

Similar problems regarding other types of eigenvalues can be found in the literature. One of them relates to predator-prey systems, where the authors are led to consider shape optimization problem with mixed Dirichlet-Neumann conditions, see \cite{darioStuff, predatorPrayStuff}. Another setting one can consider is Dirichlet-Steklov conditions, which lead to different eigenvalues. For results of this kind see for instance \cite{isoperimetricStekolv, darioStuff}

In this work we assume that $\Omega \subset \mathbb{R}^2$ has a $C^2$ boundary, and we study:
$$
    \lambda_1(\Omega - B_\delta(x))
$$
for balls $B_\delta(x) \subset \Omega$, close to the boundary. We notice that since $\partial \Omega$ is $C^2$, if the radius $\delta$ of the ball $B_\delta(x)$, is small enough, then for all $P \in \partial \Omega$, there exists $x \in \Omega$ such that $B_\delta(x) \subset \Omega$ and $\partial B_\delta(x)$ is touching $\partial \Omega$ at the point $P$. 
\begin{center}
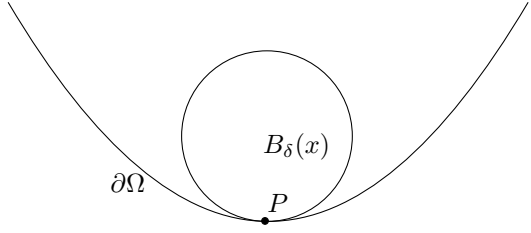
\begin{figure}
    \centering
    \begin{tikzpicture}[x=0.75pt,y=0.75pt,yscale=-0.7,xscale=0.7]

\draw   (278,77.5) .. controls (278,43.53) and (305.53,16) .. (339.5,16) .. controls (373.47,16) and (401,43.53) .. (401,77.5) .. controls (401,111.47) and (373.47,139) .. (339.5,139) .. controls (305.53,139) and (278,111.47) .. (278,77.5) -- cycle ;
\draw   (153,-19) .. controls (277.89,191.67) and (402.78,191.67) .. (527.67,-19) ;
\draw  [fill={rgb, 255:red, 0; green, 0; blue, 0 }  ,fill opacity=1 ] (335.67,138.67) .. controls (335.67,137.33) and (336.75,136.25) .. (338.08,136.25) .. controls (339.42,136.25) and (340.5,137.33) .. (340.5,138.67) .. controls (340.5,140) and (339.42,141.08) .. (338.08,141.08) .. controls (336.75,141.08) and (335.67,140) .. (335.67,138.67) -- cycle ;

\draw (335,73.4) node [anchor=north west][inner sep=0.75pt]    {$B_{\delta }( x)$};
\draw (225,102.4) node [anchor=north west][inner sep=0.75pt]    {$\partial \Omega $};
\draw (338,117.23) node [anchor=north west][inner sep=0.75pt]    {$P$};

\end{tikzpicture}
    \caption{Ball touching parabolic boundary}
\end{figure}

\end{center}
This also implies that there exists a tubular neighborhood of $\partial \Omega$, 

\begin{equation}
\label{DefOfTubNeighbour}
V_d := \{x \in \Omega: \dist(x,\partial \Omega)<d\},
\end{equation}
for $d>0$ small enough, such that for all $p \in V_d$ there exists a unique $z(p) \in \partial \Omega$ such that:
\begin{equation}
\label{tubularUniqueness}
    \min_{y \in \partial \Omega}
    \dist(p, \partial \Omega)
=
    \dist(p, z(p)),
\end{equation}
and $p-z(p)$ is colinear with the normal at $z(p) \in \partial \Omega$. This is not true for $C^{1,\alpha}$ boundaries for $\alpha \in ]0,1[$. Now define:
$$
    \dist_{min}(A, B)
:=
    \inf\{
        \dist(p, B):
        p\in A
    \},
$$
where $A$ and $B$ are two sets of $\mathbb{R}^2$.

In this work, we prove the following
\begin{thm}
\label{RealMainTheorem}
    Given a bounded domain $\Omega \subset \mathbb R^2$ with $C^2$ boundary, there exists $\epsilon := \epsilon(\Omega)>0$ such that for any $\delta \in ]0,\epsilon[$, if $p \in \Omega$ satisfies $B_\delta(p) \subset \Omega$ and
    \begin{equation}
    \label{conditionOnBallForThm}
        \dist_{min}(\partial \Omega, B_\delta(p))
    <
        \delta^2
    \end{equation}
    then there exists $v \in \mathbb R^2$ such that $B_\delta(p+tv) \subset \Omega$ for small $t>0$ and
    \begin{equation}
    \label{conclusionOfPositiveDerivative}
        \left.\frac{d}{dt}\right|_{t=0}\lambda_1(B_\delta(p+tv))>0.
    \end{equation}
    In particular $B_\delta(p)$ is not a solution of \eqref{problemSet}.
\end{thm}
More precisely, we will show that \eqref{conditionOnBallForThm} implies that $z(p) \in \partial \Omega$ is well defined, and $v$ is the unit vector from $z(p)$ to $p$. With this choice, equation \eqref{conclusionOfPositiveDerivative} implies that minimizing balls of radius small enough, must be searched away from the boundary.

Given $p$ and $\delta>0$, let $v$ from Theorem \ref{RealMainTheorem} be given by $e_y$ (we can always make a rotation of the domain since this does not change the eigenvalues or eigenfuntions). This entire work will rely on the Hadamard's formula, found in \cite[Theorem 2.5.1]{Henrot}. In particular we will show that it implies that the formula \eqref{conclusionOfPositiveDerivative} can be written as
\begin{equation}
    \left.\frac{d}{dt}\right|_{t=0}\lambda_1(B_\delta(p+te_y))
    =
        \int_{\partial B_\delta(p)}
    \frac{\partial u_{\delta,p}}{\partial y}
    \left(\frac{\partial u_{\delta,p}}{\partial \nu}\right)
    d\mathcal{H}^1,
\end{equation}
where $u_{\delta,p}$ is the normalized first Dirichlet eigenfunction of $\Omega - B_\delta(p)$. To conclude we will prove that
\begin{equation}
    \label{integralDerivativeStuff}
        \int_{\partial B_\delta(p)}
    \frac{\partial u_{\delta,p}}{\partial y}
    \left(\frac{\partial u_{\delta,p}}{\partial \nu}\right)
    d\mathcal{H}^1
    >
    0.
\end{equation}
under the conditions of Theorem \ref{RealMainTheorem}.

The first section is some preliminaries of this work. We state some of the main theorems used from references, and how we will use them. We also show some generic properties of the first eigenfunctions $u_{\delta, p} \in H^1_0(\Omega - B_\delta(p))$ that are independent of $p$ and $\delta$, such as bounds on the eigenvalues, and bounds on the eigenfunctions.

The next sections of the proof will rely in dividing $\partial B_\delta(p)$ in three different pieces $C_\theta^+$, $C_\theta^-$, $A_\theta$ given by \eqref{TopCircleDef}, \eqref{BotCircleDef}, \eqref{SidesDef} and estimating the integral \eqref{integralDerivativeStuff}, restricted to each of these set. The plan of the proof will be to analyse blowups given by:
    \begin{equation}
    \label{blowupIntro}
        \overline{u}_{n}(x,y)
    :=
        \frac{u_n(\delta_n(x,y) + p_n)
        }{\delta_n}.
    \end{equation}
for suitably chosen $p_n$ and $\delta_n$ given by a contradiction argument.
These will converge to a positive harmonic function $v:K \rightarrow \mathbb{R}$ in the set:
    \begin{equation}
    \label{upHalfPlaneWithoutBall}
        K:=\{y>-1\}-B_1(0).
    \end{equation}
Given a fixed $\theta$, we will use regularity theory to compare $\overline{u}_n$ in the sets $C_\theta^+$ and $A$. This is possible because these will correspond to smooth parts of $K$. On the other hand we will not be able to use regularity theory to study the integral in the set $C_\theta^-$, since this is close to a singularity of the set $K$. We will need the following regularity theorems

The second section is devoted to proving an integral bound for $u_{\delta,p}$ such that one can prove that the blowup limits of the sequences \eqref{blowupIntro} are non-trivial. This will be done by obtaining a lower bound on integrals of $u_{\delta,p}$ in specific sets, by comparing to harmonic functions. It will also take advantage of the fact that for the first eigenfunction $u_\Omega$ of the set $\Omega$, there exists $\Lambda > 0$ such that:
    $$
        \min_{z \in \partial \Omega}
        \frac{\partial u_\Omega(z)}{\partial \nu}
    \geq
        \Lambda.    
    $$
    This is true by application of H\"opf's lemma, and the fact that $u_\Omega \in C^1(\overline{\Omega})$. Also for small balls $B_\delta(p)$ close enough to the boundary then the eigenfunctions $u_{\delta,p} \in H^1(\Omega-B_\delta(p))$ will be close in $H^1$ to $u_\Omega$. Thus using regularity theory, for a set in the interior $V \Subset \Omega$ the functions $u_{\delta,p}$ and $u_{\Omega}$ will also be close in the $C^0(V)$ norm. This will then be used to create harmonic functions that are below $u_{\delta,p}$ and allow non-trivial lower bounds that are preserved under blowups.

    The third section is devoted to proving an upper bound for the integral:
    $$
    \left|
        \int_{C_\theta^-}
        \left(
        \frac{\partial u_{\delta,p}}{\partial \nu}
    \right)^2
    \langle
        e_y,\nu
    \rangle
    d\mathcal{H}^1
    \right|,
    $$
    which does not come directly from the blowup argument.
    The upper bound on this integral will be obtained by choosing $\theta$ small enough, and a choice of distance $d := \dist_{min}(p, \partial \Omega)$, such that there exists $C>0$ independent of $\delta$ and $p$ such that:
    \begin{equation}
    \label{upBound}
        \left|
        \int_{C_\theta^-}
        \left(
        \frac{\partial u_{\delta,p}}{\partial \nu}
    \right)^2
    \langle
        e_y,\nu
    \rangle
    d\mathcal{H}^1
    \right|
    \leq
    C
    \delta
    \theta^2.
    \end{equation}
    This will be done by taking advantage of the proximity of the balls $B_\delta(p)$, to the boundary $\partial \Omega$, in particular using the geometric fact that the height of a ball is quadratic with the angle $\theta$, when $\theta$ is close to zero.

    Finally the penultimate section is divided into two subsections. In the first one we prove that if $\delta$ is small enough and $\dist_{min}(\partial B_\delta(p),p)< \delta^2$, then:
    \begin{equation}
    \label{positiveSidesIntroduction}
        \int_{A}
        \left(
        \frac{\partial u_{\delta,p}}{\partial \nu}
    \right)^2
    \langle
        e_y,\nu
    \rangle
    d\mathcal{H}^1
    \geq
    0.
    \end{equation}
    This is done through a contradiction, and use of a blowup argument on the sequence \eqref{blowupIntro} and use of regularity theory. The blow up will be done along balls $B_{\delta_n}(p_n)$ satisfying:
    $$
        \dist_{min}(\partial B_{\delta_n}(p_n), \partial \Omega) < \delta_n^2,
    $$
    obtaining in the limit, a positive harmonic function $v : K \rightarrow \mathbb{R}$. Properties of this harmonic function are studied. In particular for the side integrals, one is interested in studying $\frac{\partial v}{\partial \theta}$. In particular the one will prove that for $x>0$ then:
    $$
        \frac{\partial v}{\partial \theta}(x,y)
        \geq 0
    $$
    By use of H\"opf's lemma, we would then have that:
    $$
        \frac{\partial^2 v}{\partial \theta \partial \nu}(x,y)
    > 0,
    $$
    from which we would be able to conclude that the side integrals for the function $v$ are positive, and so must also be for the sequence $u_n$.

    By use of the fact that $\frac{\partial v}{\partial \theta}$ is harmonic, to prove it is positive for $x>0$, one can reduce the question to proving that it is positive in an arc:
    $$
        \partial B_R(0) \cap \{x>0, y>-1\},
    $$
    and this $R$ can be as big as one wants. Thus to study $\frac{\partial v}{\partial \theta}$, one uses a blowdown of the harmonic function $v$. Given a sequence $R_n \rightarrow \infty$ one considers:
    $$
        v_n(x,y)
    :=
        \frac{1}{R_n}
        v(R_n(x,y) + (-1,0)).
    $$
    One can prove that this sequence will converge to a linear function given by $\alpha y$, for some $\alpha \geq 0$. Again using regularity theory, for $z \in \partial B_{R_n}$, we will have that $\nabla v(z)$ will be close to $\alpha e_y$, if $R_n$ is big. With this we conclude that if $Z_n = \partial \left(B_{R_n}\cap \{x>0,y>-1\}\right)$ then
    $$
        \lim_n
        \inf_{(x,y) \in Z_n}
        \frac
        {\partial v(x,y)}
        {\partial \theta}
        \geq0
    $$
    from which one can conclude the desired positivity of $\frac{\partial v}{\partial \theta}$.

    The second subsection is devoted to proving that there exists $c>0$ and $\tilde{\delta}>0$ such that for $\delta < \tilde{\delta}$ and $p \in \Omega$ satisfying $\dist_{min}(\partial B_\delta(p), \partial \Omega) < \delta^2$, then:
    \begin{equation}
    \label{introTopIntegral}
        \int_{C_\theta^+}
        \frac{\partial u_{\delta,p}}{\partial y}
        \left(
        \frac{\partial u_{\delta,p}}{\partial \nu}
        \right) \geq c\theta \delta.
    \end{equation}
This will be done again by contradiction, using a blowup argument with limit $v$ which is a nontrivial positive harmonic function. Then through use of H\"opf's lemma and regularity theory we will be able to arive at a contradiction, concluding \eqref{introTopIntegral}. We will use formulas \eqref{upBound}, \eqref{positiveSidesIntroduction} and \eqref{introTopIntegral} to prove \eqref{integralDerivativeStuff}.

\subsubsection*{Acknowledgements}
The author is supported by the Research Foundation – Flanders (FWO) via the Odysseus programme Geometric and analytic properties of metric measure spaces with spectral curvature constraints, with applications to manifold learning (G0DBZ23N). Special thanks to professor Hugo Tavares for helping with the writting at an initial stage, and David Tewodrose for helping with the writting and revisions of the final work.

\section{Preliminaries}
    
    \subsection{Notation}
    \begin{itemize}
        \item $\Omega \subset \mathbb R^2$ a bounded open connected set with $C^2$ boundary.
        \item $u_\Omega:\Omega \rightarrow \mathbb R$ the normalized first Dirichlet eigenfunction of $\Omega$.
    \end{itemize}
    Given $\delta>0$ and $p \in \Omega$ such that $B_\delta(p) \subset \Omega$, let
    \begin{itemize}
        \item $\lambda_{\delta,p}$ be the first eigenvalue of $\Omega - B_\delta(p)$. \item $u_{\delta,p}$ the first normalized Dirichlet eigenfunction of $\Omega - B_\delta(p)$.
        \item $V_d = \{x\in \Omega : d(x,\partial \Omega) < d\}.$
        \item $
            \dist_{min}(A, B)
        =
            \inf\{
                \dist(p, B):
                p\in A
            \}.$, tubular neighborhood of $\partial \Omega$.
        \item $C_\theta^+
        =
            \left\{z \in \partial B_\delta(p):
            \frac{\langle z-p, e_y \rangle}{|z-p|}
            \geq
            \cos(\theta)
            \right\}.$
        \item $C_\theta^-
        =
            \left\{
            z \in \partial B_\delta(p):
            \frac{\langle z-p, e_y \rangle}{|z-p|}
            \leq
            -
            \cos(\theta)
            \right\}.$
        \item $A_\theta = \left\{
            z \in \partial B_\delta(p):
            -\cos(\theta)
            \leq
            \frac{\langle z-p, e_y \rangle}{|z-p|}
            \leq
            \cos(\theta)
        \right\}.$
        \item $K=
        \{
            (x,y) \in \mathbb{R}^2:
            y > -1
        \}    
    -
        B_{1}(0).$
        \item $z:V_d \rightarrow \partial \Omega$ gives unique point such that $\min_{y \in \partial \Omega}
    \dist(p, \partial \Omega)
=
    \dist(p, z(p))$.
        \item $Q_{\delta, \theta, p}
        :=
            \{
                (x,y) \in \Omega- B_\delta(p):
                -\delta \sin(\theta)\leq x \leq \delta \sin(\theta);
                \quad
                -\tilde{C}\delta^2
                \leq
                y
                \leq
                d_{\delta,p}
                +
                \delta- \delta \cos(\theta)
            \}.$
        \item $P_{M, \delta, d}(p)
        :=
            \left\{(x,y):
                (y-(p_y+\delta))
                \geq
                M(x-p_x)^2 
                \text{ and }
                (y-(p_y+\delta)) \leq \frac{3}{2}d
            \right\}$.
            \item$
                \text{Top}_{M,\delta, d}(p)
            :=
                P_{M,\delta, d}(p) \cap \left\{ (y-(p_y+\delta)) = \frac{3}{2}d\right\},
            $
            \item$
            Q_{\delta}(p)
        :=
            [p_x-\delta, p_x+\delta]
            \times
            [p_y,p_y+2\delta].
            $
            \item$
                \overline{Q}_{\delta}(p)
            :=
                [p_x-\delta,p_x+\delta]
                \times
                [p_y+\frac{3}{2}\delta, p_y+2\delta]
            $
            \item $\mathbb{H} = \{(x,y) \in \mathbb{R}^2: y>0\}$ 
    \end{itemize}
    
    \subsection{Hadamard's formula}

        \begin{thm}\cite[Theorem 2.5.1]{Henrot}
            Let $\Omega \subset \mathbb R^N$ be an open connected bounded subset with $C^2$ boundary. Given $V$ a smooth vector field, and $\Phi(t) := Id_{\mathbb{R}^N} + tV$, define $\Omega_t := \Phi(t)(\Omega)$. Let $u_t \in H^1_0(\Omega_t)$ be the first eigenfunction of the Dirichlet Laplacian such that $||u_t||_{L^2(\Omega_t)} = 1$. Then $\lambda_1(\Omega_t)$ is differentiable at $t = 0$ and:
            $$
                \left. \frac{d}{dt}\right|_{t=0}
                \lambda_1(\Omega_t) 
            =
                -\int_{\partial \Omega}
                \left(
                    \frac{\partial u}{\partial \nu}
                \right)^2
                \langle
                    V,
                    \nu
                \rangle
                d \mathcal{H}^{N-1}.
            $$
            where $\nu$ is the exterior normal.
        \end{thm}
        Consider balls of small enough radius such that $B_\delta(p) \Subset V_d$. Let $V$ be a vector field such that $V = \frac{p-z(p)}{|p-z(p)|}$ in a neighborhood a neighborhood of $\partial B_\delta(p)$, and $V = 0$ in $\partial \Omega$. By taking $\Phi(t) = Id_{\mathbb{R}^2} + tV$, we have that: 
        $$
            \Phi(t)(\Omega - B_\delta(p))
        =
            \Omega - B_\delta(p + t\frac{p-z(p)}{|p-z(p)|}).
        $$
        By a simple rotation we can assume that $\frac{p-z(p)}{|p-z(p)|}$ is the vector $e_y$ (Throughout the proof we make this assumption, since the rotation does not change the eigenvector). Then we have by Hadamard's formula
        \begin{equation}
        \label{DerivativeEqualToSomeIntegral}
            \left.\frac{d}{dt}\right|_{t=0}
            \lambda_1\left(
            \Phi(t)
            (\Omega - B_\delta(p))
            \right)
            =
            \int_{\partial B_\delta(p)}
            \left(\frac{\partial u_{\delta,p}}{\partial \nu}\right)^2
            \langle
                e_y,\nu
            \rangle
            d\mathcal{H}^1
            =
            \int_{\partial B_\delta(p)}
            \frac{\partial u_{\delta,p}}{\partial y}
            \left(\frac{\partial u_{\delta,p}}{\partial \nu}\right)
            d\mathcal{H}^1,
        \end{equation}
        where $\nu$ is the exterior normal to $\partial B_\delta(p)$ (which is interior to $\Omega-B_\delta(p)$). The equality above follows since $u_{\delta,p}|_{\partial B_\delta(p)} = 0$, thus $\nabla u_{\delta,p}(x) = \frac{\partial u_{\delta,p}}{\partial \nu}(x) \nu$ for $x \in \partial B_\delta(p)$. As such we have that for $x\in \partial B_\delta(p)$
        \begin{equation}
            \frac{\partial u_{\delta,p}}{\partial y}(x)
        =
            |\nabla u(x)|
            \langle
                e_y,
                \nu
            \rangle
        =
            \left(\frac{\partial u_{\delta,p}}{\partial \nu}(x)\right)
            \langle
                e_y,
                \nu
            \rangle.
        \end{equation}
        Given an angle $\theta \in ]0, \frac{\pi}{2}[$, define the sets given by
\begin{equation}
\label{TopCircleDef}
    C_\theta^+
:=
    \left\{z \in \partial B_\delta(p):
    \frac{\langle z-p, e_y \rangle}{|z-p|}
    \geq
    \cos(\theta)
    \right\},
\end{equation}
\begin{equation}
\label{BotCircleDef}
    C_\theta^-
:=
    \left\{
    z \in \partial B_\delta(p):
    \frac{\langle z-p, e_y \rangle}{|z-p|}
    \leq
    -
    \cos(\theta)
    \right\},
\end{equation}
\begin{equation}
    \label{SidesDef}
    A_\theta := \left\{
            z \in \partial B_\delta(p):
            -\cos(\theta)
            \leq
            \frac{\langle z-p, e_y \rangle}{|z-p|}
            \leq
            \cos(\theta)
        \right\}.
\end{equation}
an image of these is provided by figure \eqref{decompOfCircle}.
    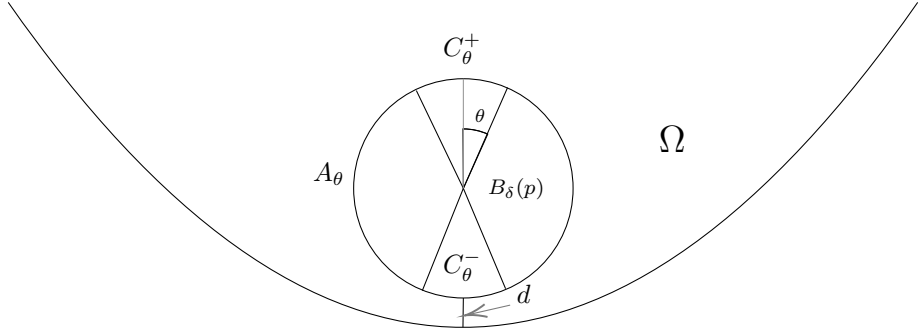
\begin{figure}
    \centering
        
    \begin{tikzpicture}[x=0.75pt,y=0.75pt,yscale=-1,xscale=1]
    
    \draw   (101,42) .. controls (253.89,260.67) and (406.78,260.67) .. (559.67,42) ;
    \draw   (275.2,135.87) .. controls (275.2,105.31) and (299.97,80.53) .. (330.53,80.53) .. controls (361.09,80.53) and (385.87,105.31) .. (385.87,135.87) .. controls (385.87,166.43) and (361.09,191.2) .. (330.53,191.2) .. controls (299.97,191.2) and (275.2,166.43) .. (275.2,135.87) -- cycle ;
    \draw    (352.83,85) -- (330.53,135.87) ;
    \draw    (306.83,86.5) -- (330.53,135.87) ;
    \draw    (330.53,135.87) -- (309.83,187.5) ;
    \draw    (330.53,135.87) -- (351.83,186.5) ;
    \draw  [draw opacity=0] (330.3,105.87) .. controls (330.38,105.87) and (330.46,105.87) .. (330.53,105.87) .. controls (334.92,105.87) and (339.08,106.81) .. (342.83,108.5) -- (330.53,135.87) -- cycle ; \draw   (330.3,105.87) .. controls (330.38,105.87) and (330.46,105.87) .. (330.53,105.87) .. controls (334.92,105.87) and (339.08,106.81) .. (342.83,108.5) ;  
    \draw [color={rgb, 255:red, 151; green, 151; blue, 151 }  ,draw opacity=1 ]   (330.53,135.87) -- (330.53,80.53) ;
    \draw    (330.53,191.2) -- (330.33,206) ;
    \draw [color={rgb, 255:red, 122; green, 122; blue, 122 }  ,draw opacity=1 ]   (354.13,194.8) -- (333.28,199.56) ;
    \draw [shift={(331.33,200)}, rotate = 347.15] [color={rgb, 255:red, 122; green, 122; blue, 122 }  ,draw opacity=1 ][line width=0.75]    (10.93,-3.29) .. controls (6.95,-1.4) and (3.31,-0.3) .. (0,0) .. controls (3.31,0.3) and (6.95,1.4) .. (10.93,3.29)   ;
    
    \draw (341.7,129.4) node [anchor=north west][inner sep=0.75pt]  [font=\footnotesize]  {$B_{\delta }( p)$};
    \draw (427.8,103) node [anchor=north west][inner sep=0.75pt]  [font=\Large]  {$\Omega $};
    \draw (319,56.23) node [anchor=north west][inner sep=0.75pt]    {$C_{\theta }^{+}$};
    \draw (335,94.4) node [anchor=north west][inner sep=0.75pt]  [font=\scriptsize]  {$\theta $};
    \draw (319,165.73) node [anchor=north west][inner sep=0.75pt]    {$C_{\theta }^{-}$};
    \draw (253.5,121.73) node [anchor=north west][inner sep=0.75pt]    {$A_\theta$};
    \draw (356.28,183.4) node [anchor=north west][inner sep=0.75pt]    {$d$};

    \end{tikzpicture}
    \caption{Image of decomposition of the circle}
    \label{decompOfCircle}
    \end{figure}
    The study of the integral \eqref{DerivativeEqualToSomeIntegral} will be done by restricting it to each of these and prove the positivity there.

    \subsection{Elliptic regularity}
        \begin{thm}
        \label{ClassicReg}
            Let $V \Subset W \subset \Omega$ be open sets. There exists $C = C(V,W, m)$ such that for all $v \in H^1(\Omega)$ satisfying:
            $$
                -\Delta v = f \text{ in } W,
            $$
            then:
            $$
                ||v||_{H^m(V)}
            \leq
                C
                (
                    ||f||_{H^{m-2}(W)}
                +
                    ||v||_{L^2(W)}
                ).
            $$
        \end{thm}
        This first regularity theorem is the most classical version of an interior regularity theorem. This will be used to derive some bounds for $u_{\delta,p}$ in the interior of the set $\Omega-B_\delta(p)$ far from the boundary.

        \begin{thm}\cite{regularity}[Theorem 2.14]
        Let $\Omega \subset \mathbb{R}^n$ be an open set, and $u \in L^\infty(\Omega)$ satisfy:
        $$
            -\Delta
            u(x)
            =
            g(x)
        $$
        for $x \in \Omega$. Then $u \in C^{1,1}_{loc}(\Omega)$ and
        $$
            ||u||_{C^{1,1}(V)}
        \leq
            C
            \left(
                1
                +
                ||u||_{L^\infty(\Omega)}
                +
                ||g||_{L^\infty(\Omega)}
            \right)
        $$
        for $V\Subset \Omega$ where $C = C(\dist_{min}(V, \partial \Omega))$.
    \end{thm}
    
    \begin{thm}\cite[Theorem 2.17]{regularity}
    \label{boundaryReg}
        Let $B_1^+ = \{(x,y)\in B_1(0):y\geq0 \}$ and $u \in L^\infty(B_1^+)$ such that:
        $$
            -\Delta u(x)
            =
            g(x)
        $$
        for $x \in B_1^+$. Assume also that $u|_{\{y=0\}} = 0$ in the Sobolev trace sense. Then:
        $$
            ||D^2u||_{L^\infty(B^+_{\frac{1}{4}})}
        \leq
            C
            (1 + ||u||_{L^\infty(B_1^+)} + ||g||_{L^\infty(B_1^+)}).
        $$
    \end{thm}
    These other two theorems will be used for analysing the blowup sequence \eqref{blowupIntro} and its limit both close to $\partial B_1(0)$ and in the interior of $K$ (given by \eqref{upHalfPlaneWithoutBall}).

    \subsection{Other results}

    Using properties of how the complementary Hausdorff distance relates to the continuity of eigenvalues (see for instance \cite[Sections 4.6, 4.7]{bucur}), we know that the function:
    $$
        (\delta, p) \in ]0, \epsilon[ \times \Omega \mapsto 
        \lambda_{\delta,p} := \lambda_1(\Omega - B_\delta(p))
    $$
    is continuous. With this we can prove the first lemma.
    \begin{lemma}
        \label{boundOnEigenValues}
        There exists $\epsilon>0$, small enough and $C>0$, such that for all $\delta < \epsilon$ and $p \in \Omega$ satisfying $B_\delta(p) \subset \Omega$, then $\lambda_{\delta,p} = \lambda_1(\Omega-B_\delta(p)) < C$.
    \end{lemma}

    \begin{lemma}
    \label{ineqEigenFunc}
        If we extend $u_{\delta,p}$ by zero to $\mathbb{R}^2-\Omega$, then:
        $$
            -\Delta u_{\delta, p}
        \leq
            \lambda_{\delta,p}
            u_{\delta, p}
            \text{ in }
            \mathbb{R}^2,
        $$
        where we extend $u_{\delta,p}$ by zero to $\mathbb{R}^2-\Omega$.
    \end{lemma}
    \begin{lemma}
    \label{unifBoundLemma}
        There exists $C>0$ such that for $\delta < \epsilon$ and $p$ such that $B_\delta(p) \subset \Omega$, then:
        $$
            ||u_{\delta, p}||_{L^\infty(\Omega)}
        \leq
            C.
        $$
    \end{lemma}
    \begin{proof}
        From Lemma \ref{boundOnEigenValues} we concude that $u_{\delta,p}$ is uniformly bounded in $H^1_0(\Omega)$. Also that there exists $C>0$ such that:
        $$
            -\Delta u_{\delta,p}
        \leq
            Cu_{\delta,p}.
        $$
        By use of a Brezis-Krato argument, we derive uniform $L^\infty(\Omega)$ bounds.
    \end{proof}
    We will call the family of eigenfunctions:
    \begin{equation}
        \mathcal{F}_\delta
    =
        \{
        u_{\delta,p}:
            B_{\delta}(p) \subset \Omega,
            u_{\delta,p} \text{ first eigenfunction of }
            \Omega - B_\delta(p)
        \}
    \end{equation}
    
    Now we will estimate the derivatives of the eigenfunctions $u_{\delta,p}$ close to the boundary. For this we will use a comparison principle, and compare with the solution $v_\Omega \in H^1_0(\Omega)$ of the problem:
    $$
        -\Delta v_\Omega
    =
        1.
    $$
    This solution satisfies $v_\Omega \geq 0$ in $\Omega$. Also this solution is regular, in particular it is Lipschitz continuous in $\overline{\Omega}$, thus there exists $C>0$ such that for all $x,y \in \overline{\Omega}$ we have:
    $$
        |v_\Omega(x) - v_\Omega(y)|
    \leq
        L|x-y|.
    $$
    \begin{lemma}
        \label{LipBound}
        There exists $C>0$ such that for all $\delta < \epsilon$ and $u_{\delta, p} \in \mathcal{F}_\delta$, $x \in \Omega$ and $y \in \partial \Omega$ then:
        $$
            |u_{\delta,p}(x) - u_{\delta, p}(y)|
            =
            |u_{\delta,p}(x)|
            \leq
            C|y-x|.
        $$
        In particular for all $x \in \Omega$:
        $$
            |u_{\delta, p}(x)|
        \leq
            C\dist(x,\partial \Omega).
        $$
    \end{lemma}
    \begin{proof}
        By Lemma \ref{boundOnEigenValues} and \ref{unifBoundLemma} there exists $C$ such that for all $u_{\delta,p} \in \mathcal{F}_\delta$:
        $$
            -\Delta u_{\delta,p} \leq C.
        $$
        Thus $-\Delta u_{\delta, p } \leq C = -C\Delta v_\Omega$ in $\Omega - B_\delta(p)$. Since $Cv_\Omega|_{\Omega - B_\delta(p)}\geq 0 = u_{\delta, p }|_{\Omega - B_\delta(p)}$, by the maximum principle we conclude that:
        $$
            u_{\delta,p}
        \leq
            Cv_\Omega.
        $$
        From this, if $y \in \partial \Omega$ and $x \in \Omega$ then:
        $$
            |u_{\delta,p}(x)|
            \leq
            C|v_\Omega(x)|
            =
            C|v_\Omega(x)-v_\Omega(y)|
            \leq
            CL|x-y|.
        $$
    \end{proof}
    We remind of the definition of $V_d$ given by \eqref{DefOfTubNeighbour}. We will use the above Theorem \ref{ClassicReg} to conclude the following lemma:
    \begin{lemma}
        \label{C1Closeness}
        For every $\epsilon>0$ and $ d>0$, there exists $\tilde{\rho}>0$ such that, if $B_\delta(p)\subset V_{\tilde{\rho}}$ then:
        $$
            ||u_{\delta, p} - u_{\Omega}||_{C^1(\Omega - V_d)}
        \leq
            {\epsilon}.
        $$
    \end{lemma}
    \begin{proof}
        Given $d>0$, we suppose that $B_\delta(p) \subset V_{\frac{d}{2}}$. This implies that $u_{\delta,p}$ satisfies the equation:
        $$
        -\Delta
        u_{\delta, p}
        =
        \lambda_{\delta,p}u_{\delta,p}
        \quad
        \text{ in } 
        \Omega - V_{\frac{d}{2}}.
        $$
        Using Theorem \ref{ClassicReg} inductively, for each $m$, there exists $C$ depending on $d$ and $m$, such that:
        $$
            ||u_{\delta,p}||_{H^m(\Omega - V_{d})}
        \leq
            C(
            \lambda_{\delta,p} ||u_{\delta,p}||_{H^{m-2}(\Omega - V_{\frac{d}{2}})}
        +
            ||u_{\delta,p}||_{L^2}
            ).
        $$
        By Lemma \ref{boundOnEigenValues}, there exists $\tilde{\delta}>0$ and $C>0$ such that, if $\delta < \tilde{\delta}$, then $\lambda_{\delta,p} < C$. Using this fact and induction, we conclude that there exists $C>0$ depending on $d$ and $m$ such that:
        \begin{equation}
        \label{SobolevBoundOnFunc}
            ||u_{\delta,p}||_{H^m(\Omega - V_d)}
        \leq
            C.
        \end{equation}
        Again using Theorem \ref{ClassicReg}, since:
        $$
        -\Delta(
        u_{\delta, p}
        -
        u_\Omega)
        =
        \lambda_{\delta,p}u_{\delta,p}
        -
        \lambda_1(\Omega)u_\Omega
        \quad
        \text{ in } 
        \Omega - V_{\frac{d}{2}},
        $$
        there exists $C$ depending on $d$ and $m$, such that:
        \begin{align*}
            ||u_{\delta, p}
        -
            u_\Omega
            ||_{H^m(\Omega-V_d)}
        &\leq
            C(
            ||
            \lambda_{\delta,p}u_{\delta,p}
        -
            \lambda_1(\Omega)u_\Omega
            ||_{H^{m-2}(\Omega-V_{\frac{d}{2}})}
            +
            ||
                u_{\delta, p}
            -
                u_\Omega
            ||_{L^2(\Omega-V_{\frac{d}{2}})}
            )\\
        &\leq
        C(
        (\lambda_{\delta,p}-\lambda_1(\Omega))  ||u_{\delta,p}||_{H^{m-2}(\Omega-V_{\frac{d}{2}})}      
        +
        \lambda_1(\Omega)
        ||u_{\delta,p} - u_{\Omega}||_{H^{m-2}(\Omega-V_{\frac{d}{2}})}
        +
        ||u_{\delta,p}-u_{\Omega}||_{L^2(\Omega-V_{\frac{d}{2}})}
        ).
        \end{align*}
        We can use induction on this argument and equation \eqref{SobolevBoundOnFunc}, we can conclude that if $\delta < \tilde{\delta}$ and $B_\delta(p) \subset V_{\frac{d}{2}}$, there exists $\overline{C} = \overline{C}(d,m)$ such that:
        \begin{equation}
        \label{differenceSobolevBound}
        ||u_{\delta, p}
        -
            u_\Omega
            ||_{H^m(\Omega-V_d)}
        \leq
            \overline{C}(
                (\lambda_{\delta,p}-\lambda_1(\Omega))
                +
                ||u_{\delta,p}-u_{\Omega}||_{L^2(\Omega-V_\frac{d}{2})}
            ).
        \end{equation}

        Now we notice that by the results in \cite[Sections 4.6, 4.7]{bucur}, if 
        $$\delta_n \rightarrow 0 \text{ and } \dist(\partial \Omega, B_{\delta_n}(p_n)) \rightarrow 0,
        $$
        since the complementary Hausdorff difference:
        $$
            \mathcal{H}^c(\Omega, \Omega - B_{\delta_n}(p_n))
            \rightarrow
            0
        $$
        then:
        $$
            ||
                u_n
            -
                u_{\Omega}
            ||_{H^1(\Omega)}
            \rightarrow 0.
        $$
        Thus for all $0<\epsilon$, there exists $0<\tilde{\rho}<\min\{\frac{d}{2},\tilde{\delta}\} $ small enough, such that if:
        $$
        B_\delta(p) \subset V_{\tilde{\rho}},
        $$
        then
        $$
            ||
                u_{\delta, p}
            -
                u_{\Omega}
            ||_{H^1(\Omega)}
            <
            \frac{\epsilon}{2\overline{C}},
            \quad
            \lambda_{\delta,p}
            -
            \lambda_1(\Omega)
            <
            \frac{\epsilon}{2\overline{C}},
            \quad
            \delta < \tilde{\delta}
            .
        $$
        From equation \eqref{differenceSobolevBound} we conclude:
        $$
            ||u_{\delta, p}
        -
            u_\Omega
            ||_{H^m(\Omega - V_d)}
            \leq
            \epsilon.
        $$
        so for $m$ big enough the Sobolev embedding will imply that:
        $$
            ||u_{\delta, p}
        -
            u_\Omega||_{C^1(\Omega - V_d)}
            \leq
            \epsilon,
        $$
        so for $\epsilon$ small enough the proof follows.
    \end{proof}
    \begin{lemma}
        \label{lowerBoundOnEigenfunctions}
            There exists $\tilde{d}>0$, such that for every $0<d<\tilde{d}$, there exists $\tilde{\rho}>0$ such that, if $B_\delta(p) \subset V_{\tilde{\rho}}$, and $x \in V_{2d} - V_d$ then:
            $$
                u_{\delta, p}(x)
            \geq
                \frac{\Lambda d}{8},
            $$
            and
            $$
                \Lambda
            =
                \min_{x \in \partial \Omega}
                \left|
                    \frac
                    {\partial u_\Omega}
                    {\partial \nu}
                    (x)
                \right|
                >
                0
            $$
            where $u_\Omega \in H^1_0(\Omega)$ is the first eigenfunction of the set $\Omega$.
        \end{lemma}
        \begin{proof}
            By H\"opf's Lemma, since $\Omega$ is $C^2$, for all $p \in \partial \Omega$ we have that:
            $$
            \left(
                \frac{\partial u_\Omega}
                {\partial \nu}
            \right) (p)
            >
            0.
            $$
            Since $u_\Omega \in C^1(\overline{\Omega})$, there exists $\Lambda > 0$ such that:
            $$
                \Lambda
            =
                \min_{x \in \partial \Omega}
                \left|
                    \frac
                    {\partial u_\Omega}
                    {\partial \nu}
                    (x)
                \right|
                >
                0.
            $$
            Since $u_\Omega \in C^2(\overline{\Omega})$, there exists $\tilde{d}>0$ small enough such that, for every $d<\tilde{d}$:
            $$
                u_\Omega(x)
            \geq
                \frac{\Lambda \dist(x,\partial \Omega)}{2}
            $$
            for all $x \in V_{2d}$. In particular for $x \in V_{2d}-V_d$ we have:
            \begin{equation}
            \label{lowBoundRand1}
                u_{\Omega}(x)
            \geq
                \frac{\Lambda d}{2}.
            \end{equation}
            
            Now we use Lemma \ref{C1Closeness} to conclude that there exists $\tilde{\rho}>0$, such that if $B_\delta(p) \subset V_{\tilde{\rho}}$, and $u_{\delta,p} \in H^1_0(\Omega - B_\delta(p))$ is the first eigenfunction of $\Omega - B_\delta(p)$, then:
            \begin{equation}
            \label{unifProxEq1}
                ||
                    u_{\delta, p}
                    -
                    u_\Omega
                ||_{L^\infty(\Omega - V_d)}
                \leq
                \frac{\Lambda d}{4}.
            \end{equation}
            Using equations \eqref{lowBoundRand1} and \eqref{unifProxEq1} implies that for $x\in V_{2d}-V_d$ we have:
            $$
                u_{\delta, p}(x)
                \geq
                \frac{\Lambda d}{4}
            $$
            concluding the proof.
        \end{proof}
        \begin{lemma}
            \label{inequalityLemmaForXEq}
            For all $C>0$, there exist $\tilde{\theta}>0$ and $\tilde{\delta}>0$ such that for all $\theta < \tilde \theta$, $\delta < \tilde \delta$, if $x>0$ satisfies:
            \begin{equation}
            \label{xInequalityRandEq}
                x^2 \leq C\delta^4\theta^4 + C\delta\theta^2x
                \quad
            \Rightarrow
                \quad
                x < 2C\delta \theta^2,
            \end{equation}
            \begin{equation}
            \label{xInequlityRand20}
                x^2 \leq C\delta^4 + C\delta x
            \quad
            \Rightarrow
            \quad
                x < 2C\delta.
            \end{equation}
        \end{lemma}
        \begin{proof}
            We choose $\tilde{\delta}>0$ such that for all $\delta < \tilde \delta$ then $C\delta^2 < C^2$. If $x> 2C\delta \theta^2$ and $\delta^2 < \frac{1}{C}$, then:
            $$
                x^2-C\delta\theta^2x
            =
                x(x-C\delta\theta^2)
            >
                C^2\delta^2\theta^4
            >
                C\delta^4\theta^4
            $$
            and so
            $$
                x^2 > C\delta^4\theta^4 + C\delta\theta^2x.
            $$
            Thus if inequality \eqref{xInequalityRandEq} is satisfied, then $x < 2C\delta\theta^2$.

            On the other hand if $x > 2C\delta$ then:
            $$
                x(x-C\delta)>2C^2\delta^2.
            $$
            If $\delta^2 < 2C$ then:
            $$
                x(x-C\delta)>2C^2\delta^2 > C\delta^4,
            $$
            and so $x^2\leq C\delta^4 +C\delta x$ can not be satisfied. This proves the second part of the statement.
        \end{proof}

\section{Upper bound of bottom integral}

        This section is devoted to proving the following proposition:
        \begin{prop}
        \label{bottomThm}
            There exist $\theta_0 := \theta_0(\Omega)>0$ and $C = C(\theta_0)$ such that for all $\theta < \theta_0$ there exists $\tilde{\delta} := \tilde{\delta}(\Omega, \theta_0, \theta)>0$ such that if $\delta < \tilde{\delta}$ and $\dist(B_\delta(p), \partial \Omega) < \delta^2$ given the set:
            \begin{equation}
                \label{bottomCircle}
                C_\theta^-
            :=
                \{
                    z \in \partial B_\delta(p):
                    \langle
                        \nu_z,
                        e_y
                    \rangle
                    < 
                    -\cos(\theta)
                \},
            \end{equation}
            then:
            $$
                \left|
                    \int_{C_\theta^-}
                    \frac
                    {\partial u_{\delta, p}}
                    {\partial \nu}
                    \frac
                    {\partial u_{\delta,p}}
                    {\partial y}
                    d\mathcal{H}^1
                \right|
            \leq
                C
                \delta
                \theta^2.
            $$
        \end{prop}
        \begin{proof}
        We assume without loss of generality that the interior unit normal at $z(p)$ is $e_y$ and that $z(p) = (0,0)$. Also by the fact that the boundary is $C^2$, there exists a $\tilde{\delta}$ and $\tilde{C}>0$ independent of $p$, such that for all $\delta < \tilde{\delta}$:
        \begin{equation}
            [-\delta, \delta]
            \times [-\delta, \delta]
            \cap
            \Omega
            -
            B_\delta(p)
        \subset
            [-\delta, \delta]
            \times
            [-\tilde{C}\delta^2, \delta],
        \end{equation}
        and that the interior normal $\nu_z$ for $z \in \partial \Omega \cap [-\delta, \delta]
            \times [-\delta, \delta]$ satisfies:
        \begin{equation}
            \label{interiorNormalStuff}
            \langle
                \nu_z,
                e_y
            \rangle
            >
            0.
        \end{equation}

\begin{figure}
    \centering
            \begin{tikzpicture}[x=0.75pt,y=0.75pt,yscale=-1,xscale=1]
\draw  [color={rgb, 255:red, 0; green, 0; blue, 0 }  ,draw opacity=1 ][fill={rgb, 255:red, 200; green, 200; blue, 200 }  ,fill opacity=1 ] (249.68,156.33) -- (405.18,156.33) -- (405.18,250.83) -- (249.68,250.83) -- cycle ;
\draw   (250.27,135.57) .. controls (250.27,92.9) and (284.86,58.31) .. (327.53,58.31) .. controls (370.21,58.31) and (404.8,92.9) .. (404.8,135.57) .. controls (404.8,178.25) and (370.21,212.84) .. (327.53,212.84) .. controls (284.86,212.84) and (250.27,178.25) .. (250.27,135.57) -- cycle ;
\draw [color={rgb, 255:red, 0; green, 0; blue, 0 }  ,draw opacity=0.54 ]   (90,228.31) -- (570.67,228.97) ;
\draw  [fill={rgb, 255:red, 0; green, 0; blue, 0 }  ,fill opacity=1 ] (327.14,228.64) .. controls (327.14,226.87) and (328.57,225.44) .. (330.33,225.44) .. controls (332.1,225.44) and (333.53,226.87) .. (333.53,228.64) .. controls (333.53,230.4) and (332.1,231.83) .. (330.33,231.83) .. controls (328.57,231.83) and (327.14,230.4) .. (327.14,228.64) -- cycle ;
\draw [color={rgb, 255:red, 130; green, 130; blue, 130 }  ,draw opacity=1 ]   (90.53,251.2) -- (568.93,250.4) ;
\draw   (48,250.44) .. controls (235.56,221.37) and (423.11,221.37) .. (610.67,250.44) ;
\draw    (195.17,233.17) -- (186.22,193.22) ;
\draw [shift={(185.78,191.27)}, rotate = 77.37] [color={rgb, 255:red, 0; green, 0; blue, 0 }  ][line width=0.75]    (10.93,-3.29) .. controls (6.95,-1.4) and (3.31,-0.3) .. (0,0) .. controls (3.31,0.3) and (6.95,1.4) .. (10.93,3.29)   ;

\draw (253.5,213.4) node [anchor=north west][inner sep=0.75pt]  [font=\footnotesize]  {$z( p) =( 0,0)$};
\draw (414.32,198.07) node [anchor=north west][inner sep=0.75pt]  [font=\footnotesize]  {$[ -\delta ,\ \delta ] \times \left[ -\tilde{C} \delta ^{2}, \delta\right]$};
\draw (338.45,91.51) node [anchor=north west][inner sep=0.75pt]    {$B_{\delta }( p)$};
\draw (180.05,135.9) node [anchor=north west][inner sep=0.75pt]  [font=\Large]  {$\Omega $};
\draw (197.95,183.5) node [anchor=north west][inner sep=0.75pt]    {$\nu _{z}$};

\end{tikzpicture}
\caption{}
\end{figure}
        Let:
        $$
            d_{\delta, p}
        :=
            \text{dist}
            (\partial \Omega, B_\delta(p)).
        $$
        Define the set:
        $$
            Q_{\delta, \theta, p}
        :=
            \{
                (x,y) \in \Omega- B_\delta(p):
                -\delta \sin(\theta)\leq x \leq \delta \sin(\theta);
                \quad
                -\tilde{C}\delta^2
                \leq
                y
                \leq
                d_{\delta,p}
                +
                \delta- \delta \cos(\theta)
            \}.
        $$
Notice that if $\theta_0>0$ is small enough, if $\theta < \theta_0$, then:
        $$
            1-\cos(\theta)
            <
            2\theta^2.
        $$
        If $d_{\delta, p} < \delta^2$ this implies that:
        \begin{equation}
        \label{QcontainedInSquaredRectangle}
            Q_{\delta,\theta,p}
        \subset
            [-\delta \sin(\theta), \delta \sin(\theta)]
            \times
            [-\tilde{C} \delta^2, \delta^2 + \delta 2\theta^2].
        \end{equation}
        \begin{claim}
        \label{BottomConta}
        There exist $\tilde{\theta}$ and $\tilde{\delta}>0$ such that if $\delta < \tilde{\delta}$ and $\theta < \tilde{\theta}$ then:
            $$
            \left|
                \int_{C_\theta^-}
                \frac{\partial u_{\delta, p}}
                {\partial y}
                \frac{\partial u_{\delta, p}}
                {\partial \nu}
                d
                \mathcal{H}^1
            \right|
            \leq
            2
            \int_{Q_{\delta, \theta, p} \cap 
            \{x = -\delta\sin(\theta)\}}
            |\nabla u_{\delta, p}|^2
            dy
            +
            2\int_{Q_{\delta, \theta, p} \cap \{x = \delta \sin(\theta)\}}
            |\nabla u_{\delta, p}|^2
            dy.
            $$
        \end{claim}

            Let $S_1 = \{x=-\delta \sin(\theta)\} \cap Q_{\delta, \theta, p}$ and $S_2 = \{x=\delta \sin(\theta)\} \cap Q_{\delta, \theta, p}$.
            Using integration by parts we obtain that:
            \begin{align}
                \lambda_{\delta, p}\int_{Q_{\delta, \theta, p}}
                u_{\delta, p} 
                \frac
                {\partial u_{\delta, p}}
                {\partial y}
               &=
                -\int_{Q_{\delta, \theta, p}}
                \div(\nabla 
                \frac{\partial u_{\delta, p}}
                {\partial y}
                )
                u_{\delta, p}\\
                \label{randomIntegrationByPartsEq}
                &=
                \int_{Q_{\delta, \theta, p}}
                \langle
                    \nabla(
                    \frac{\partial u_{\delta, p}}
                    {\partial y})
                    ,
                    \nabla u_{\delta,p}
                \rangle
                -
                \int_{S_1 \cup S_2}
                u_{\delta, p}
                \langle
                    \nabla 
                    \frac{\partial u_{\delta, p}}
                    {\partial y}
                    ,
                    \nu
                \rangle.
            \end{align}
            Since $-\div(\nabla \frac{\partial u_{\delta,p}}{\partial y}) = \lambda_{\delta,p}\frac{\partial u_{\delta, p}}{\partial y}$, applying integration by parts again to \eqref{randomIntegrationByPartsEq}, we obtain:
            $$
                \lambda_{\delta, p}\int_{Q_{\delta, \theta, p}}
                u_{\delta, p} 
                \frac
                {\partial u_{\delta, p}}
                {\partial y}
            =
                \lambda_{\delta, p}\int_{Q_{\delta, \theta, p}}
                u_{\delta, p} 
                \frac
                {\partial u_{\delta, p}}
                {\partial y}
                +
                \int_{\partial Q_{\delta, \theta, p}}
                \frac{\partial u_{\delta,p}}{\partial y}
                \langle
                    \nabla
                    u_{\delta, p}
                    ,
                    \nu
                \rangle
                -
                \int_{S_1 \cup S_2}
                u_{\delta, p}
                \langle
                    \nabla 
                    \frac{\partial u_{\delta, p}}
                    {\partial y}
                    ,
                    \nu
                \rangle,
            $$
            equivalently:
            \begin{equation}
            \label{equalityAfterTwoIntByParts}
                \int_{\partial Q_{\delta, \theta, p}}
                \frac{\partial u_{\delta,p}}{\partial y}
                \langle
                    \nabla
                    u_{\delta, p}
                    ,
                    \nu
                \rangle
                -
                \int_{S_1 \cup S_2}
                u_{\delta, p}
                \langle
                    \nabla 
                    \frac{\partial u_{\delta, p}}
                    {\partial y}
                    ,
                    \nu
                \rangle=0
            \end{equation}
            We have that:
            $$
                \partial Q_{\delta, \theta, p}
                =
                S_1
                \cup
                S_2
                \cup
                C_\theta^-
                \cup
                (Q_{\delta, \theta, p} \cap \partial \Omega).
            $$
            When $\nu$ is the exterior normal to $\partial \Omega$ at $(Q_{\delta, \theta, p} \cap \partial \Omega)$, we know there exists $\tilde{\delta}>0$ such that, if $\delta < \tilde{\delta}$, then equation \eqref{interiorNormalStuff} is satisfied. This implies that since $u_{\delta,p}$ is positive that:
            $$
                \langle
                \nabla
                    u_{\delta, p}
                    ,
                    \nu
                \rangle
                \leq
                0,
                \quad
                \frac{\partial u_{\delta,p}}{\partial y}(x,y) \geq 0 \quad \forall (x,y) \in Q_{\delta, \theta, p}\cap \partial \Omega.
            $$
            Thus we conclude that:
            \begin{equation}
            \label{positiveIneqForBoundaryIntegral}
                \int_{(Q_{\delta, \theta, p} \cap \partial \Omega)}
                \frac{\partial u_{\delta,p}}{\partial y}
                \langle
                    \nabla
                    u_{\delta, p}
                    ,
                    \nu
                \rangle
                \leq
                0.
            \end{equation}
            If $\nu_z$ is the exterior normal of $Q_{\delta,\theta,p}$, then for $z\in S_1$, we have that:
            $$
                \langle
                    \nabla 
                    \frac{\partial u_{\delta, p}}
                    {\partial y}
                    ,
                    \nu_z
                \rangle
                =
                -
                \frac{\partial^2 u_{\delta, p}}
                {\partial x\partial y},
            $$
            and for $z \in S_2$:
            $$
                \langle
                    \nabla 
                    \frac{\partial u_{\delta, p}}
                    {\partial y}
                    ,
                    \nu_z
                \rangle
                =
                \frac{\partial^2 u_{\delta, p}}
                {\partial x\partial y}.
            $$
                
            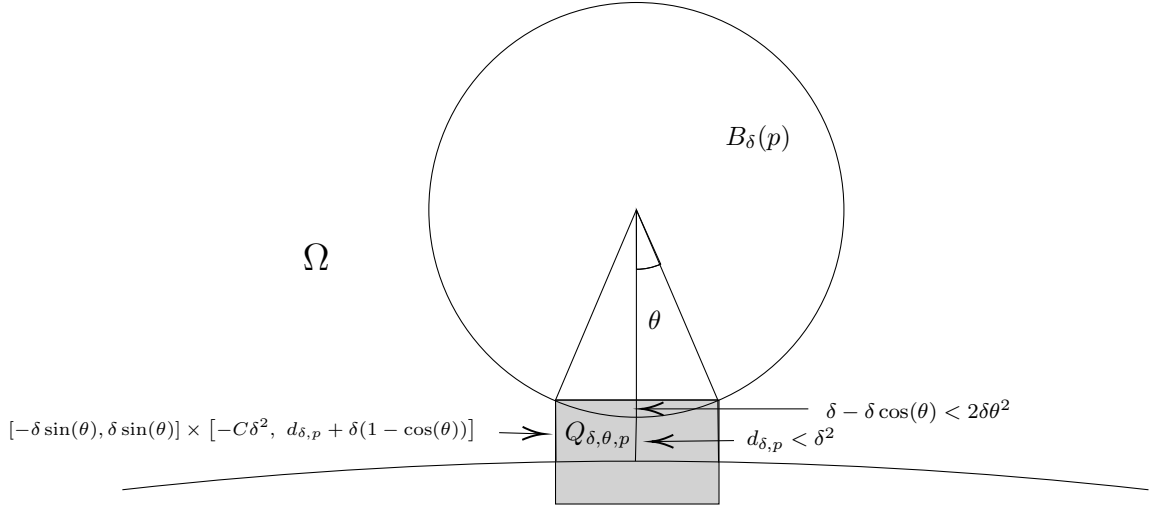
\begin{figure}[h]
    \centering
            \begin{tikzpicture}[x=0.75pt,y=0.75pt,yscale=-1,xscale=1]

\draw  [fill={rgb, 255:red, 210; green, 210; blue, 210 }  ,fill opacity=1 ] (289.33,238.22) -- (371.78,238.22) -- (371.78,290.69) -- (289.33,290.69) -- cycle ;
\draw   (225.44,142.22) .. controls (225.44,84.42) and (272.31,37.56) .. (330.11,37.56) .. controls (387.92,37.56) and (434.78,84.42) .. (434.78,142.22) .. controls (434.78,200.03) and (387.92,246.89) .. (330.11,246.89) .. controls (272.31,246.89) and (225.44,200.03) .. (225.44,142.22) -- cycle ;
\draw    (330.11,142.22) -- (289.33,238.22) ;
\draw    (330.11,142.22) -- (330.11,246.89) ;
\draw    (330.11,142.22) -- (371.17,238.17) ;
\draw  [draw opacity=0] (342.36,169.61) .. controls (338.64,171.28) and (334.51,172.21) .. (330.17,172.22) -- (330.11,142.22) -- cycle ; \draw   (342.36,169.61) .. controls (338.64,171.28) and (334.51,172.21) .. (330.17,172.22) ;  
\draw   (70.87,283.5) .. controls (243.37,264.23) and (415.88,264.23) .. (588.38,283.5) ;
\draw    (289.56,269.11) -- (289.33,238.22) ;
\draw    (371.56,269.78) -- (371.17,238.17) ;
\draw    (329.63,269.05) -- (330.11,246.89) ;
\draw [color={rgb, 255:red, 0; green, 0; blue, 0 }  ,draw opacity=0.27 ]   (379.33,258.67) -- (339.83,259.14) ;
\draw [shift={(337.83,259.17)}, rotate = 359.31] [color={rgb, 255:red, 0; green, 0; blue, 0 }  ,draw opacity=0.27 ][line width=0.75]    (10.93,-3.29) .. controls (6.95,-1.4) and (3.31,-0.3) .. (0,0) .. controls (3.31,0.3) and (6.95,1.4) .. (10.93,3.29)   ;
\draw [color={rgb, 255:red, 0; green, 0; blue, 0 }  ,draw opacity=0.26 ]   (289.33,238.22) -- (371.17,238.17) ;
\draw [color={rgb, 255:red, 0; green, 0; blue, 0 }  ,draw opacity=0.36 ]   (416.83,242.67) -- (334.83,242.67) ;
\draw [shift={(332.83,242.67)}, rotate = 360] [color={rgb, 255:red, 0; green, 0; blue, 0 }  ,draw opacity=0.36 ][line width=0.75]    (10.93,-3.29) .. controls (6.95,-1.4) and (3.31,-0.3) .. (0,0) .. controls (3.31,0.3) and (6.95,1.4) .. (10.93,3.29)   ;
\draw    (260.89,254.67) -- (281.56,255.27) ;
\draw [shift={(283.56,255.33)}, rotate = 181.68] [color={rgb, 255:red, 0; green, 0; blue, 0 }  ][line width=0.75]    (10.93,-3.29) .. controls (6.95,-1.4) and (3.31,-0.3) .. (0,0) .. controls (3.31,0.3) and (6.95,1.4) .. (10.93,3.29)   ;

\draw (334.83,191.95) node [anchor=north west][inner sep=0.75pt]    {$\theta $};
\draw (373.83,98.01) node [anchor=north west][inner sep=0.75pt]    {$B_{\delta }( p)$};
\draw (292.83,248.23) node [anchor=north west][inner sep=0.75pt]    {$Q_{\delta ,\theta, p}$};
\draw (384.5,250.9) node [anchor=north west][inner sep=0.75pt]  [font=\footnotesize]  {$d_{\delta, p}< \delta ^{2}$};
\draw (424.6,237.04) node [anchor=north west][inner sep=0.75pt]  [font=\footnotesize]  {$\delta -\delta \cos( \theta ) < 2\delta \theta ^{2}$};
\draw (12,246.07) node [anchor=north west][inner sep=0.75pt]  [font=\scriptsize]  {$[ -\delta \sin( \theta ) ,\delta \sin( \theta )] \times \left[ -C\delta ^{2} ,\ d_{\delta, p}+\delta ( 1-\cos( \theta ))\right]$};
\draw (160.67,158.4) node [anchor=north west][inner sep=0.75pt]  [font=\Large]  {$\Omega $};

\end{tikzpicture}
\caption{Image of $Q_{\delta,\theta,p}$}
\label{fig:enter-label}
\end{figure}

            Since:
            $$
                -\int_{S_1 \cup S_2}
                u_{\delta, p}
                \langle
                    \nabla 
                    \frac{\partial u_{\delta, p}}
                    {\partial y}
                    ,
                    \nu
                \rangle=
                \int_{S_1}
                    u_{\delta,p}
                    \frac{\partial^2 u_{\delta,p}}
                    {\partial y \partial x}
                    dy
                -
                \int_{S_2}
                    u_{\delta,p}
                    \frac{\partial^2 u_{\delta,p}}
                    {\partial y \partial x}
                    dy
            $$
            and
            $$
                \int_{\partial Q_{\delta, \theta, p}}
                \frac{\partial u_{\delta,p}}{\partial y}
                \langle
                    \nabla
                    u_{\delta, p}
                    ,
                    \nu
                \rangle
            =
                \int_{C_\theta^-}
                \frac
                {\partial u_{\delta,p}}
                {\partial y}
                \frac
                {\partial u_{\delta, p}}
                {\partial \nu}
            -
                 \int_{S_1}
                    \frac
                    {\partial u_{\delta,p}}
                    {\partial y}
                    \frac
                    {\partial u_{\delta,p}}
                    {\partial x}dy
            +
                \int_{S_2}
                    \frac
                    {\partial u_{\delta,p}}
                    {\partial y}
                    \frac
                    {\partial u_{\delta,p}}
                    {\partial x}dy
            +
                \int_{(Q_{\delta, \theta, p} \cap \partial \Omega)}
                \frac{\partial u_{\delta,p}}{\partial y}
                \langle
                    \nabla
                    u_{\delta, p}
                    ,
                    \nu
                \rangle
            $$
            substituting inequality \eqref{positiveIneqForBoundaryIntegral} in \eqref{equalityAfterTwoIntByParts} we obtain:
            $$
            0
            \leq
                -\int_{C_\theta^-}
                \frac
                {\partial u_{\delta,p}}
                {\partial y}
                \frac
                {\partial u_{\delta, p}}
                {\partial \nu}
            \leq
                \int_{S_1}
                \left(
                    u_{\delta,p}
                    \frac{\partial^2 u_{\delta,p}}
                    {\partial y \partial x}
                -
                    \frac
                    {\partial u_{\delta,p}}
                    {\partial y}
                    \frac
                    {\partial u_{\delta,p}}
                    {\partial x}
                \right)dy
                +
                \int_{S_2}
                \left(
                -
                    u_{\delta,p}
                    \frac{\partial^2 u_{\delta,p}}
                    {\partial y \partial x}
                +
                    \frac
                    {\partial u_{\delta,p}}
                    {\partial y}
                    \frac
                    {\partial u_{\delta,p}}
                    {\partial x}
                \right)dy.
            $$
            Applying integration by parts in the 1 dimensional integrals in $S_i$, we also have that:
            $$
                \int_{S_i}
                u_{\delta,p}
                \frac
                {\partial^2 u_{\delta, p}}
                {\partial y \partial x}
                dy
            =
                [u_{\delta,p} \frac{\partial u_{\delta,p}}{\partial x}]_{\partial S_1}
            -
                \int_{S_i}
                \frac
                    {\partial u_{\delta,p}}
                    {\partial x}
                    \frac
                    {\partial u_{\delta,p}}
                    {\partial y}
                    dy
            =
                -
                \int_{S_i}
                \frac
                    {\partial u_{\delta,p}}
                    {\partial x}
                    \frac
                    {\partial u_{\delta,p}}
                    {\partial y}
                    dy
            $$
            since for $z \in \partial S_i$, $u_{\delta,p}(z) = 0$.  From this we obtain
            $$
                -\int_{C_\theta^-}
                \frac
                {\partial u_{\delta,p}}
                {\partial y}
                \frac
                {\partial u_{\delta, p}}
                {\partial \nu}
            \leq
            -2
                \int_{S_1}
                    \frac
                    {\partial u_{\delta,p}}
                    {\partial x}
                    \frac
                    {\partial u_{\delta,p}}
                    {\partial y}
                    dy
                +
                2
                \int_{S_2}
                    \frac
                    {\partial u_{\delta,p}}
                    {\partial x}
                    \frac
                    {\partial u_{\delta,p}}
                    {\partial y}dy.
            $$
            If $\tilde{\theta}< \frac{\pi}{2}$, and $\theta<\tilde{\theta}$, then $\int_{C_\theta^-}
                \frac
                {\partial u_{\delta,p}}
                {\partial y}
                \frac
                {\partial u_{\delta, p}}
                {\partial \nu} \leq 0$, and so we have:
            $$
                \left|
                \int_{C_\theta^-}
                \frac{\partial u_{\delta, p}}
                {\partial y}
                \frac{\partial u_{\delta, p}}
                {\partial \nu}
                d
                \mathcal{H}^1
            \right|
            \leq
            2
            \int_{Q_{\delta, \theta, p} \cap 
            \{x = -\delta\sin(\theta)\}}
            |\nabla u_{\delta, p}|^2
            dy
            +
            2\int_{Q_{\delta, \theta, p} \cap \{x = \delta \sin(\theta)\}}
            |\nabla u_{\delta, p}|^2
            dy,
            $$
            finishing the proof.

        \begin{claim}
        \label{boundQuadraticNablaLemma}

            There exist $\tilde{\theta}>0$ and $C>0$ such that for all $\theta < \tilde{\theta}$, there exists $\tilde{\delta}(\theta, \Omega)$ such that if $\delta < \tilde{\delta}$ and $\dist(B_\delta(p), \partial \Omega)< \delta^2$, then:
            $$
                \int_{Q_{\delta, \theta, p}}
                |\nabla u_{\delta,p}|^2
                \leq
                C\delta^2\theta^4
            $$
        \end{claim}
            Consider a non-negative function $\phi \in C^\infty$ such that:
        $$        
            \phi|_{Q_{\delta, \theta, p}}
        =
            1, \quad 0\leq \phi \leq 1$$
        \begin{equation}
        \label{suppCondition}
        supp(\phi)
        \subset [-2\delta \sin(\theta), 2\delta \sin(\theta)]
        \times 
        [-2\delta \sin(\theta), 2\delta \sin(\theta)],
        \end{equation}
        \begin{equation}
        \label{derivativeBoundTest}
            |\nabla \phi|
            \leq
            \frac{C}{\delta \theta}.
        \end{equation}
        
            By testing $u_{\delta, p}$ with $u_{\delta,p}\phi^2 \in H^1_0(\Omega - B_\delta(p))$ we obtain:
            \begin{equation}
            \label{mainIneqForNablaBoundBottomTest}
                \int
                \langle
                \nabla u_{\delta,p}
                ,
                \nabla u_{\delta,p}
                \rangle
                \phi^2
            =
                \int 
                \langle
                \nabla u_{\delta,p},
                \nabla (u_{\delta, p} \phi^2)
                \rangle
            -
            2
                \int 
                \langle
                    \nabla u_{\delta,p},
                    \nabla \phi
                \rangle
                u_{\delta,p}
                \phi
            \leq
                \lambda_{\delta,p} 
                \int
                u_{\delta,p}^2
                \phi^2
            +
            2
                \int
                |\nabla u_{\delta,p}|
                |\nabla \phi|
                u_{\delta,p}\phi.
            \end{equation}
            There exists $\tilde{\delta}>0$, $\tilde{\theta}$ and $\tilde{C}>0$ such that if $\delta<\tilde{\delta}$ and $\theta < \tilde{\theta}$ such that if $\dist(B_\delta(p), \partial \Omega) < \delta^2$, then equation \eqref{QcontainedInSquaredRectangle} is satisfied, and from \eqref{suppCondition}, we conclude that for any $z \in supp(\phi) \cap \Omega - B_\delta(p)$ we have that:
            \begin{equation}
            \label{distToBoundaryBeQuadratic}
                \dist(z, \partial \Omega)
            \leq
                \tilde{C}\delta^2
                +
                2\delta \theta^2.
            \end{equation}
            In particular we have that:
            \begin{equation}
            \label{measureOfSupportOfStuuf}
                |supp(\phi) \cap \Omega - B_\delta(p)|
                \leq
                (4\delta \theta)
                (\tilde{C}\delta^2 + 2\delta \theta^2).
            \end{equation}
            since $|\sin(\theta)|<\theta$, and:
            $$
                supp(\phi) \cap \Omega - B_\delta(p)
            \subset 
                [-2\delta\sin(\theta), 2\delta \sin(\theta)]
                \times [-\Tilde{C}\delta^2, \delta^2 + 2\delta\theta^2].
            $$
            By Lemma \ref{LipBound} and equation \eqref{distToBoundaryBeQuadratic}, there exists $C>0$ such that, for all $z \in supp(\phi) \cap \Omega - B_\delta(p)$ we have:
            \begin{equation}
            \label{upBoundInBottomPart123}
                u_{\delta, p}(z)
            \leq
                C(\tilde{C}\delta^2 + 2\delta \theta^2).
            \end{equation}
            Using Lemma \ref{boundOnEigenValues} and equations \eqref{measureOfSupportOfStuuf} and \eqref{upBoundInBottomPart123}, we conclude that there exists $C>0$ such that:
            \begin{equation}
            \label{ineqForNormalInregral}
                \lambda_{\delta,p} 
                \int
                u_{\delta,p}^2
                \phi^2
            \leq
                C
                (\tilde{C}\delta^2+2\delta \theta)^3
                4\delta \theta.
            \end{equation}
            With the above and equation \eqref{derivativeBoundTest}, we also have that there exists $C>0$ such that:
            $$
                \int 
                |\nabla \phi|^2 u_{\delta,p}^2
            \leq
            C
                \frac{4\delta \theta
                (\tilde{C}\delta^2+2\delta\theta^2)^3}
                {(\delta \theta)^2}
            \leq
                C
                \frac{
                (\tilde{C}\delta^2+2\delta\theta^2)^3}
                {(\delta \theta)}.
            $$
           Now choose $\tilde{\delta}$ small enough depending on $\theta$ such that $\tilde{\delta} < \theta^2$. If $\delta < \tilde{\delta}$, then we have that for some $C>0$:
            \begin{equation}
            \label{ineqForutimesNablaTest}
                \int 
                |\nabla \phi|^2 u_{\delta,p}^2
            \leq
                C\frac{\delta^3 \theta^6}{\delta \theta}
            \leq
                C\delta^2 \theta^5
            \leq
                C\delta^2 \theta^4.
            \end{equation}
            for $\theta<1$. Substituting \eqref{ineqForNormalInregral} and \eqref{ineqForutimesNablaTest} in \eqref{mainIneqForNablaBoundBottomTest} implies:
            \begin{align*}
                \int
                \langle
                \nabla u_{\delta,p}
                ,
                \nabla u_{\delta,p}
                \rangle
                \phi^2
            &\leq
                C
                (\delta^2+2\delta \theta)^3
                2\delta \theta
            +
                2
                (
                    \int
                    |\nabla u_{\delta, p}|^2
                    \phi^2
                )^\frac{1}{2}
                (
                    \int
                    |\nabla \phi|^2
                    u_{\delta,p}^2
                )^\frac{1}{2}\\
            &\leq
                C\delta^4\theta^4
                +
                C
                \delta\theta^2
                (\int
                \langle
                \nabla u_{\delta,p}
                ,
                \nabla u_{\delta,p}
                \rangle
                \phi^2)^\frac{1}{2}
            \end{align*}
            Let $x = (\int
                \langle
                \nabla u_{\delta,p}
                ,
                \nabla u_{\delta,p}
                \rangle
                \phi^2)^\frac{1}{2}$. Then we have that:
                $$
                x^2
                \leq
                C\delta^4\theta^4
                +
                C\delta\theta^2
                x.
                $$
                By Lemma \ref{inequalityLemmaForXEq}, there exists $\tilde \theta$ and $\tilde \delta$ such that if $\delta < \tilde \delta$ and $\theta < \tilde \theta$ such that:
                $$
                    (\int
                \langle
                \nabla u_{\delta,p}
                ,
                \nabla u_{\delta,p}
                \rangle
                \phi^2)^\frac{1}{2}
                \leq
                2C\delta \theta^2.
                $$
                This finishes the proof.
        \begin{claim}
        \label{sideIntegralBound} 
        Let $\tilde{\delta}(\theta, \Omega)$ be the constant from Claim \ref{boundQuadraticNablaLemma}. There exist $\theta_0>0$ and $C>0$ (independent of $\theta$ and $\delta$) such that, for all $\theta \in ]0,\theta_0[$ and $\delta < \tilde{\delta}(\theta, \Omega)$, there exists $\tilde{\theta} \in ]{\theta}, 2{\theta}[$ such that:
            $$
                \int_{Q_{\delta, \tilde{\theta},p} \cap \{x= -\delta \sin(\tilde{\theta})\}}
                |\nabla u_{\delta, p}|^2
                dy
            \leq
                C\delta \theta^2;
            $$
            $$
                \int_{Q_{\delta, \tilde{\theta},p} \cap \{x= \delta \sin(\tilde{\theta})\}}
                |\nabla u_{\delta, p}|^2
                dy
            \leq
                C\delta \theta^2;
            $$
        \end{claim}
            Let:
            $$
                g(s):=
                \int_{Q_{\delta, 2\theta} \cap \{x=s\}}
                |\nabla u_{\delta,p}|^2dy.
            $$
            Then:
            $$
                \int_{Q_{\delta, 2\theta}}
                |\nabla u_{\delta, p}|^2
            =
                \int_{-\delta \sin(2\theta)}^{\delta \sin(2\theta)}
                g(s)
                ds.
            $$
            For $\delta < \tilde{\delta}(\theta, \Omega)$, by Lemma \ref{boundQuadraticNablaLemma}, there exists $C>0$ independent of $\delta$ and $\theta$ such that
            $$
                \int_{Q_{\delta, 2\theta}}
                |\nabla u_{\delta, p}|^2
            =
                \int_{-\delta \sin(2\theta)}^{\delta \sin(2\theta)}
                g(s)
                ds
                \leq
                C\delta^2\theta^4.
            $$
            Thus we have that:
            $$
                |
                \{
                s \in [-\delta \sin(2\theta), \delta \sin(2\theta)]:
                g(s)
                \geq
                C\delta \theta^2
                \}
                |
                \leq
                \delta \theta^2.
            $$
            Thus if $\theta_0$ is small enough, by a measure argument (since $|[\delta\sin(\theta),\delta\sin(2\theta)]|\geq \frac{1}{4}\theta > \delta \theta^2$ for small $\theta_0$), there must exist $\tilde{\theta} \in ]\theta,2\theta[$ such that:
            $$
                g(\delta\sin(\tilde{\theta}))
                \leq
                C\delta \theta^2;
                \quad
                g(-\delta\sin(\tilde{\theta}))
                \leq
                C\delta \theta^2.
            $$
            This shows the claim.

        To conclude the proof of Proposition \ref{bottomThm}, by Claim \ref{sideIntegralBound}, there exist $\theta_0$ and $C>0$, such that for all $\theta < \theta_0$ and $\delta < \tilde{\delta}(\Omega, 2\theta)$, there exists $\tilde{\theta} \in ]\theta, 2\theta[$, such that:
            $$
            \int_{Q_{\delta, \tilde{\theta}} \cap \{x= -\delta \sin(\tilde{\theta})\}}
                |\nabla u_{\delta, p}|^2
                dy
            \leq
                C\delta \theta^2,
            $$
            $$\int_{Q_{\delta, \tilde{\theta}} \cap \{x=\delta \sin(\tilde{\theta})\}}
                |\nabla u_{\delta, p}|^2
                dy
            \leq
                C\delta \theta^2.$$
            Applying Claim \ref{BottomConta} we conclude:
            $$
            \left|
                \int_{C_{\tilde{\theta}}^-}
                    \frac
                    {\partial u_{\delta, p}}
                    {\partial \nu}
                    \frac
                    {\partial u_{\delta,p}}
                    {\partial y}
                    d\mathcal{H}^1
            \right|
                \leq
                4C\delta \theta^2
            $$
            Since $\theta < \tilde{\theta}$ we have:
            $$
            \left|
                \int_{C_{\theta}^-}
                    \frac
                    {\partial u_{\delta, p}}
                    {\partial \nu}
                    \frac
                    {\partial u_{\delta,p}}
                    {\partial y}
                    d\mathcal{H}^1
            \right|
            \leq
            \left|
                \int_{C_{\tilde{\theta}}^-}
                    \frac
                    {\partial u_{\delta, p}}
                    {\partial \nu}
                    \frac
                    {\partial u_{\delta,p}}
                    {\partial y}
                    d\mathcal{H}^1
            \right|
                \leq
                4C\delta \theta^2
            $$
            concluding the proof.
        \end{proof}

    \section{Non trivial integral for Blowup}

        Let $p = (p_x,p_y) \in \Omega$. For any $M, \delta, d >0$, we define: 
        $$
            P_{M, \delta, d}(p)
        :=
            \left\{(x,y):
                (y-(p_y+\delta))
                \geq
                M(x-p_x)^2 
                \text{ and }
                (y-(p_y+\delta)) \leq \frac{3}{2}d
            \right\},
        $$
        $$
                \text{Top}_{M,\delta, d}(p)
            :=
                P_{M,\delta, d}(p) \cap \left\{ (y-(p_y+\delta)) = \frac{3}{2}d\right\},
        $$
        $$
            Q_{\delta}(p)
        :=
            [p_x-\delta, p_x+\delta]
            \times
            [p_y,p_y+2\delta],
        $$
        $$
            \overline{Q}_{\delta}(p)
        :=
            [p_x-\delta,p_x+\delta]
            \times
            [p_y+\frac{3}{2}\delta, p_y+2\delta].
        $$
        For this section we will need the following fact regarding sets with $C^2$ boundary.
        \begin{lemma}
        \label{geometryParabola}
            There exists $\tilde{d}>0$ and $M>0$ such that, for all $d<\tilde{d}$, if $B_\delta(p) \subset V_{\frac{d}{4}}$, 
            and the interior normal at the point $z(p) \in \partial \Omega$ is given by $e_y$, then:
            $$\Top_{M,\delta,d}(p) \subset V_{2d}-V_d.
            $$
        \end{lemma}
    
        \begin{figure}
            \centering
\tikzset{every picture/.style={line width=0.75pt}} 

\begin{tikzpicture}[x=0.75pt,y=0.75pt,yscale=-0.8,xscale=0.8]
\draw  [fill={rgb, 255:red, 210; green, 210; blue, 210 }  ,fill opacity=1 ] (250,80) .. controls (301.33,234.67) and (352.67,234.67) .. (404,80) ;
\draw   (24.67,114) .. controls (227.33,302) and (430,302) .. (632.67,114) ;
\draw   (24.33,-25) .. controls (226.11,192.33) and (427.89,192.33) .. (629.67,-25) ;
\draw   (302,221) .. controls (302,207.19) and (313.19,196) .. (327,196) .. controls (340.81,196) and (352,207.19) .. (352,221) .. controls (352,234.81) and (340.81,246) .. (327,246) .. controls (313.19,246) and (302,234.81) .. (302,221) -- cycle ;
\draw    (250,80) -- (404,80) ;

\draw (308,212.4) node [anchor=north west][inner sep=0.75pt]    {$B_{\delta }( p)$};
\draw (140,165.4) node [anchor=north west][inner sep=0.75pt]    {$V_{d}$};
\draw (134,42.4) node [anchor=north west][inner sep=0.75pt]    {$V_{2d}$};
\draw (299,105.4) node [anchor=north west][inner sep=0.75pt]    {$P_{M,\delta, d }( p)$};
\draw (319,53.4) node [anchor=north west][inner sep=0.75pt]    {$\text{Top}_{M,\delta, d }( p)$};

\end{tikzpicture}
    \caption{Image of $\text{Top}_{M,\delta, d }( p)$ and $P_{M,\delta, d }( p)$}
    \label{fig:enter-label}
\end{figure}
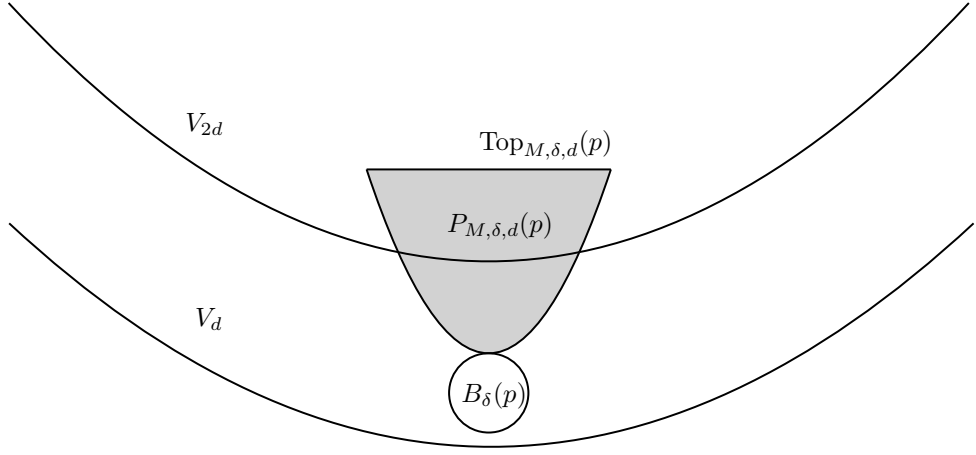

    The aim of this section is to prove the following:
    \begin{prop}
        There exists $M, \tilde d >0$ such that for all $d < \tilde d$, there exist $\tilde \sigma >0$ and $\tilde \rho >0$ such that if $B_\delta(p) \subset V_{\tilde \rho}$, then
        \begin{equation}
        \label{propLowBoundStuffing}
            \tilde{\sigma} \delta^4
    \leq
        \int_{\overline{Q}_{\delta}(p)}
        u_{\delta,p}^2,
        \end{equation}
    \end{prop}

    \begin{proof}

    Given $M, \delta, d>0$ consider $v_{\delta,p,d} \in H^1 (P_{M,\delta, d}(p))$ the smooth solution in $P_{M,\delta, d}(p)$ and satisfying the boundary conditions:
    \begin{equation}
        \Delta v_{\delta,p,d}
    =
        0,
    \quad
    \text{ on }
    \quad
    P_{M,\delta,d}(p)
    \end{equation}
    \begin{equation}
    \label{}
        v_{\delta,p, d}
            =
        0,
        \quad
        \text{ on }\quad
        \{(y-(p_y+\delta))
                =
                M(x-p_x)^2 \}
    \end{equation}
    $$
        v_{\delta,p,d}(x,y)
        =
        \frac{\Lambda d}{8}
        \cos\left(
            \frac{\pi}{2} \sqrt{\frac{2M}{3d}}(x-x_0)
        \right),
        \quad
        \text{ on }\quad
        \text{Top}_{M,\delta,d}(p).
    $$
    \begin{claim}
    \label{claimOfIneqOfHarmAndSol}
        There exist $M$, $\tilde \rho > 0$ such that if $B_\delta(p) \subset V_{\tilde \rho}$, then:
        \begin{equation}
            v_{\delta,p,d}
        \leq
            u_{\delta,p}
        \quad
        \text{ on }
        P_{M,\delta,d}(p)
        \end{equation}
    \end{claim}
Notice that due to Lemma \ref{lowerBoundOnEigenfunctions}, given $\tilde d>d>0$, we can choose $\tilde{\rho}>0$ such that if $B_\delta(p) \subset V_{\tilde{\rho}}$, then:
$$
    v_{\delta,p,d}|_{\text{Top}_{M,\delta,d}(p)}
\leq
    \frac{\Lambda d}{8}
\leq
    u_{\delta, p}|_{\text{Top}_{M,\delta,d}(p)}.
$$
Thus
$$
    v_{\delta,p,d}|_{\partial P_{M,\delta,d}(p)}
\leq
    u_{\delta,p}|_{\partial P_{M,\delta,d}(p)}
$$
and
$$
    -\Delta(u_{\delta,p}-v_{\delta,p,d}) \geq 0
$$
in $P_{M,\delta,d}(p)$, and so by maximum principles, we conclude that:
\begin{equation}
\label{inequalityEigenWithHarmonic}
    v_{\delta,p,d}|_{P_{M,\delta,d}(p)}
\leq
    u_{\delta,p}|_{P_{M,\delta,d}(p)}.
\end{equation}
We fix such a $d>0$ for the rest of this section.

Since $P_{M,\delta,d}(p)$ is simply a translation of $P_{M,d} = P_{M,d}(0,0)$, and if $v: P_{M,d} \rightarrow \mathbb{R}$ is the harmonic function having the boundary conditions above, we have that $v_{\delta,p,d}$ is a translation of $v$, that is
\begin{equation}
\label{translationOfHarmonicAux}
    v_{\delta,p,d}(z)
=
    v(z - (p_x,p_y+\delta)).
\end{equation}

\begin{claim}
\label{NonTrivialityLemmaBlowUp}
    For $M, \tilde d, \tilde \rho$ as in Claim \ref{claimOfIneqOfHarmAndSol} and $d < \tilde{d}$, there exist $\tilde{\sigma} >0$ and $\tilde{\rho}>0$, such that if $B_\delta(p) \subset V_{\tilde{\rho}}$, then:
$$
    \int_{P_{M,\delta,d}(p)\cap Q_{\delta}(p)}
        v_{\delta,p,d}^2
    \geq
        \tilde{\sigma} \delta^4.
$$
\end{claim}
By the observation of equation \eqref{translationOfHarmonicAux}, we only need to prove that there exists $\tilde \sigma >0$ such that:
$$
    \int_{P_{M,d}}
    v^2
\geq
    \tilde \sigma \delta^4.
$$
We can apply H\"opf's Lemma to conclude that there exists $c>0$, depending on $M$ and $d$ such that:
    \begin{equation}
    \label{lowBoundOnHarmonicAuxiliar}
        \frac{\partial v}
        {\partial y}
        (0,0)
        \geq
        c.
    \end{equation}
    Using $C^2$ continuity of $v$ in $P_{M,d}$ close to $(0,0)$, there exists $\tilde{\delta}>0$ small enough, such that for $\delta < \tilde{\delta}$, we have that:
    \begin{equation}
    \label{lowerBoundOnHarmonicParabola}
        \inf_{z \in P_{M,d} \cap [-\delta,\delta]^2}
        \frac{\partial v}
        {\partial y}(z)
        \geq \frac{c}{2}.
    \end{equation}
    With this, for $(x_0,y_0) \in P_{M, d}(p)\cap [-\delta,\delta]^2$ satisfying $y_0 \geq \frac{\delta}{2}$, we obtain using \eqref{lowerBoundOnHarmonicParabola}:
    \begin{equation}
    \label{integralLowBoundEqParabola}
        v(x_0,y_0)
    =
        \int_{Mx_0^2}^{y_0}
        \frac{
        \partial v(x_0,y)
        }{\partial y}
        dy
    \geq
    \frac{c}{2}
    (y_0 - Mx_0^2)
    \geq
        \frac{c}{2}(\frac{\delta}{2} - Mx_0^2).
    \end{equation}
    Choose $\tilde{\delta}>0$ small enough satisfying $M\tilde{\delta}< \frac{1}{4}$, such that for $\delta < \tilde{\delta}$ and $x_0 \in [-\delta, \delta]$ we have:
    \begin{equation}
    \label{smallParabolaStuff}
        Mx_0^2
        \leq
        M\delta^2
        \leq M\tilde{\delta} \delta
        \leq  \frac{\delta}{4},
    \end{equation}
    Using \eqref{integralLowBoundEqParabola} and \eqref{smallParabolaStuff} we conclude that for $(x_0,y_0) \in [-\delta, \delta] \times [\frac{1}{2}\delta, \delta]$:
    \begin{equation}
        \label{lowBoundOnHarmonicStuff}
         v(x_0,y_0)
         \geq
         \frac{c\delta}{8}.
    \end{equation}
    \begin{figure}
        \centering
\tikzset{every picture/.style={line width=0.75pt}} 

\begin{tikzpicture}[x=0.75pt,y=0.75pt,yscale=-1,xscale=1]

\draw  [fill={rgb, 255:red, 210; green, 210; blue, 210 }  ,fill opacity=1 ] (286.56,134.22) -- (375.68,134.22) -- (375.68,223.35) -- (286.56,223.35) -- cycle ;
\draw  [fill={rgb, 255:red, 210; green, 210; blue, 210 }  ,fill opacity=1 ] (61.45,-46.77) .. controls (241.23,253.97) and (421.01,253.97) .. (600.78,-46.77) ;
\draw  [draw={rgb, 255:red, 0; green, 0; blue, 0 }  ,fill opacity=1 ] (286.56,134.22) -- (375.68,134.22) -- (375.68,223.35) -- (286.56,223.35) -- cycle ;
\draw  [fill={rgb, 255:red, 170; green, 170; blue, 170 }  ,fill opacity=1 ] (286.56,134.22) -- (375.68,134.22) -- (375.68,157.11) -- (286.56,157.11) -- cycle ;

\draw   (17.56,205.33) .. controls (225.21,296.62) and (432.87,296.62) .. (640.53,205.33) ;
\draw   (286.73,223.17) .. controls (286.73,198.66) and (306.6,178.78) .. (331.12,178.78) .. controls (355.63,178.78) and (375.5,198.66) .. (375.5,223.17) .. controls (375.5,247.68) and (355.63,267.56) .. (331.12,267.56) .. controls (306.6,267.56) and (286.73,247.68) .. (286.73,223.17) -- cycle ;

\draw (136,12.07) node [anchor=north west][inner sep=0.75pt]    {$P_{M,\delta,d }( p)$};
\draw (376.67,202.07) node [anchor=north west][inner sep=0.75pt]    {$Q_{\delta }( p)$};
\draw (288.56,160.51) node [anchor=north west][inner sep=0.75pt]  [font=\scriptsize]  {$P_{M,\delta, d }( p) \cap Q_{\delta }( p)$};
\draw (266,100.96) node [anchor=north west][inner sep=0.75pt]  [font=\scriptsize]  {$\ \overline{Q}_{\delta}(p) = [ -\delta+p_x ,\ \delta +p_x] \ \times \left[ p_{y} +\frac{3}{2} \delta ,\ p_{y} \ +\ 2\delta \right]$};

\end{tikzpicture}
    \caption{Image of $Q_\delta(p)$ and $\overline{Q}_\delta(p)$}
    \end{figure}
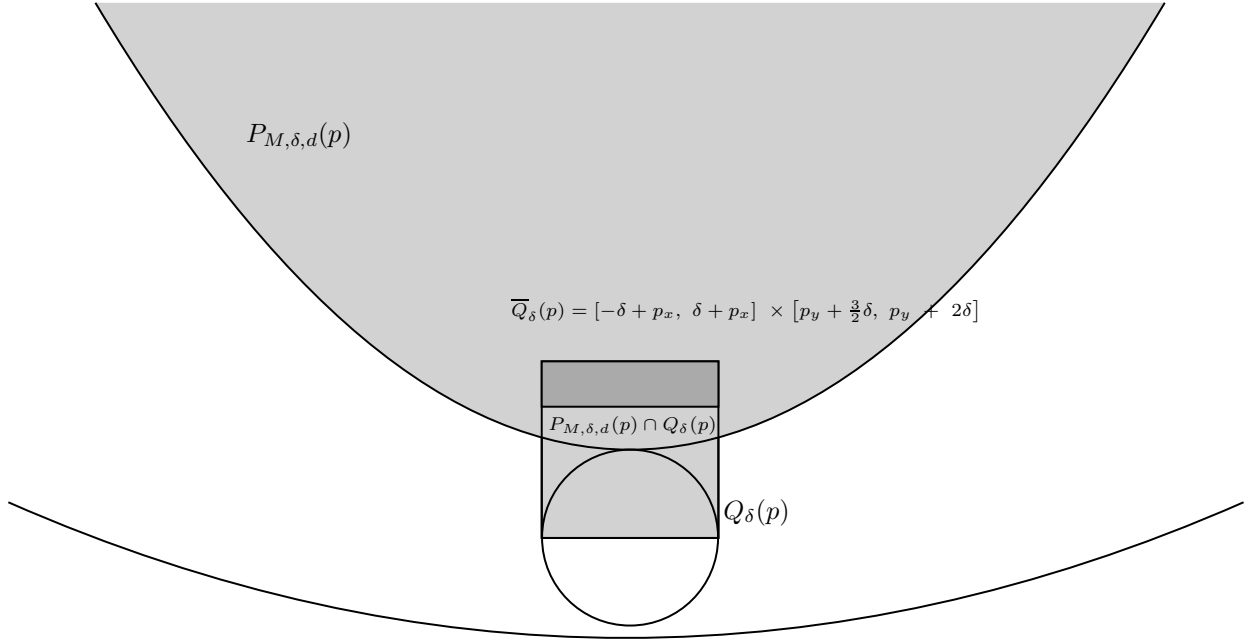
    Since $P_{M, d}(p)$ is a parabola, there exists $\tilde{\delta}>0$ small enough such that for all $\delta < \tilde{\delta}$ and $x < \delta$, we have:
    $$
        Mx^2
    <
        \frac{1}{2}\delta,
    $$
    that is
    \begin{equation}
    \label{rectangleInParabola}
        [ -\delta , \delta ]  \times [\frac{1}{2}\delta, \delta ]
        \subset 
        P_{M,d}(p) \cap [-\delta, \delta]^2.
    \end{equation}
    Using \eqref{lowBoundOnHarmonicStuff} and \eqref{rectangleInParabola}, we conclude that there exists $\tilde{\sigma} = \frac{1}{2}\frac{c^2}{16} = \frac{c^2}{32} >0$, such that:
    $$
        \int_{\overline{Q}_{\delta}(p)}
        v_{\delta,p,d}^2
    =
        \int_{[-\delta, \delta] \times [\frac{1}{2}\delta, \delta]}
        v^2
    \geq
    \tilde{\sigma} \delta^4.
    $$
    There exists $\tilde{\rho}>0$ small enough such that, if $B_\delta(p) \subset V_{\tilde{\rho}}$, then $\delta<\tilde{\delta}$ and equation \eqref{inequalityEigenWithHarmonic} is satisfied.

Applying Claims \ref{claimOfIneqOfHarmAndSol} and \ref{NonTrivialityLemmaBlowUp}, we directly obtain \eqref{propLowBoundStuffing}.
\end{proof}

\section{Blow Up Argument}
This section is divided in two subsections. The first one is dedicated to prove a positivity result for the side integrals \eqref{positiveSidesIntroduction}, that is in the set \eqref{SidesDef}. This is done by a contradiction argument, by constructing a blowup and studying it to achieve contradiction. In the first subsection most of the blowup construction is done. In the second section one proves a result about the top integral in \eqref{introTopIntegral}, also by contradiction using the blowup.

\subsection{Side Integrals}
Let $\theta_0 \in ]0, \frac{\pi}{2}[$ be a given angle, and define the sets:
$$
    A_1
:=
    \{z \in \partial B_\delta(p) : 
        -\cos(\theta_0)
        <
        \langle z-p, e_y \rangle
        <
        \cos(\theta_0);
        \langle z-p, e_x \rangle
        >0
    \};
$$
$$
    A_2
:=
    \{z \in \partial B_\delta(p) : 
        -\cos(\theta_0)
        <
        \langle z-p, e_y \rangle
        <
        \cos(\theta_0);
        \langle z-p, e_x \rangle
        <0
    \}.
$$
In this section, we will prove the following theorem: 

\begin{prop}
\label{positiveIntegrals}
    There exists $\tilde{\theta}_0 > 0$ such that for all $\theta_0 < \tilde{\theta}_0$, there exists $\tilde{\epsilon} := \tilde{\epsilon}(\Omega, \theta_0)$, such that if:
$$
    \delta < \tilde{\epsilon};
\quad
    \dist_{min}(B_\delta(p), \partial \Omega)
<
    \delta^2,
$$
then:
$$
    \int_{A_1}
                \frac{\partial u_{\delta, p}}
                {\partial \nu}
                \frac{\partial u_{\delta, p}}
                {\partial y}
                d\mathcal{H}^1
\geq
    0,
$$
$$
    \int_{A_2}
                \frac{\partial u_{\delta, p}}
                {\partial \nu}
                \frac{\partial u_{\delta, p}}
                {\partial y}
                d\mathcal{H}^1
\geq
    0.
$$
\end{prop}

To prove this, we will proceed by contradiction. So we suppose there exists a sequence $u_n:=u_n$ such that for all $n$ we have:
\begin{equation}
\label{contradictHyp}
     \int_{A_1}
                \frac{\partial u_n}
                {\partial \nu}
                \frac{\partial u_n}
                {\partial y}
                d\mathcal{H}^1
<
0;
\quad
\delta_n\rightarrow 0;
\quad
d_n:=\dist_{min}(\partial B_{\delta_n}(p_n), \partial \Omega)
<
\delta_n^2.
\end{equation}
We consider a blowup of this sequence given by:
\begin{equation}
\label{blowUpDefinition}
    \overline{u}_{n}(x)
:=
    \frac
    {u_{n}}
    {\delta_n}(\delta_n x + p_n).
\end{equation}
We suppose without loss of generality for all the section that $p_n = 0$ and that $e_y = \frac{p_n - z(p_n)}{|p_n - z(p_n)|}$, where $z(p)$ is defined in section 2.1. Recall that the set $K$ defined by equation \eqref{upHalfPlaneWithoutBall}.
\begin{lemma}
\label{harmonicFunction}
    There exists a non-trivial $v \in H^1_{loc}(K)$ such that up to a subsequence:
    $$
        \overline{u}_n
    \rightharpoonup_{H^1_{loc}(K)}
        v.
    $$
    Moreover, $v$ is harmonic in $K$, $v\geq 0$, $v|_{\partial K} = 0$. For any $K' \Subset int(K)$ compactly contained, there exists $C= C(K')$ such that:
    $$
        \sup_n
        ||
            \nabla \overline{u}_n
        ||_{L^2(K')}
        \leq
        C,
    $$
    and for any $R>0$,
    $$
        \sup_n ||u_n||_{L^\infty(B_R(0))} < +\infty
    $$
\end{lemma}
\begin{proof}
    First we notice that the sets $\frac{1}{\delta_n}(\Omega - B_{\delta_n}(p_n))$ will converge to $K$ due to hypothesis \eqref{contradictHyp} and the fact that $\partial \Omega$ is $C^2$. In particular for any compact set $K'\subset int(K)$, there exists $N$ big enough such that for all $n> N$ we have:
    $$
        K'
    \subset 
    \frac{1}{\delta_n}(\Omega - B_{\delta_n}(p_n)).
    $$
    Also we prove that:
    $$
        \sup_n
        ||\overline{u}_n||_{H^1(K')}
    <
        +\infty.
    $$
    This will prove that there exists $v \in H^1_{loc}(K)$, and a subsequence such that $u_n \rightharpoonup_{H^1_{loc}(K)} v$.

    By Lemma \ref{LipBound}, since $z(p_n) \in \partial \Omega$, there exists $C>0$ such that:

    $$
        \overline{u}_n(x)
    =
        \frac
    {u_n}
    {\delta_n}(\delta_n x + p_n)
    \leq
       \frac{1}{\delta_n} 
       C|\delta_n x + (p_n-z(p_n))|
    \leq
        C|x|
        +
        2C,
    $$
    where we use the fact that $|p_n-z(p_n)| \leq 2\delta_n$ for $n$ big enough. This implies that $\overline{u}_n$ is uniformly bounded in $L^\infty(K')$ and so it also is in $L^2(K')$ since it is compact.

    Now we will prove that for all $K' \subset K$, there exists a subsequence such that:

    $$
        \sup_n
        ||\nabla \overline{u}_n||_{L^2(K')}
    <
        +\infty.
    $$
To prove this we consider a test function $\phi \in C_c^\infty(K)$ such that for some $C>0$:
    $$\phi|_{K'}=1; \quad |\nabla \phi| \leq C; \quad 0 \leq \phi \leq 1, \quad |\{\phi>0\}|<C.$$
    Thus considering the rescaling $\phi_n(x):= \phi( {x}/{\delta_n})$ we have that:
    \begin{equation}
        \label{testFuncCond}
        |\nabla \phi_n| \leq \frac{C}{\delta_n};
        \quad
        |\{\phi_n>0\}|
        =
        |\{\phi>0\}|\delta_n^2
        \leq
        C\delta_n^2
    \end{equation}
    Now we test the equation of $u_n$ against $u_{n} \phi_n^2$ to obtain:
    \begin{align}
        \int_{\Omega}
        |\nabla u_n|^2\phi_n^2
        dx
    &=
        \lambda_{\delta_n,p_n}
        \int_{\Omega}
        u_n^2
        \phi_n^2
        dx
    -
        2\int_{\Omega}
        \langle
            \nabla u_n
        ,
            \nabla \phi_n
        \rangle
        u_n\phi_n
        dx\\
    \label{randEquationWithHolderAux}
    &\leq
        \lambda_{\delta_n,p_n}
        \int_{\Omega}
        u_n^2
        \phi_n^2
        dx
    +
        2
        ||\phi_n \nabla u_n||_{L^2(\Omega)}
        ||u_n \nabla \phi_n||_{L^2(\Omega)},
    \end{align}
    where in the last inequality we used H\"older's inequality. Now by the conditions of $\phi_n$ in \eqref{testFuncCond}, Lemma \ref{boundOnEigenValues} and Lemma \ref{LipBound}, there exists some $C>0$ such that, for all $z \in \{\phi_n u_n>0\}$ we have $u_n(z) \leq C\delta_n$, $\lambda_{\delta_n,p_n}\leq C$ and \eqref{testFuncCond} are satisfied, and so we conclude that there exists $M = \max\{C^3, C^\frac{5}{2}\}>0$ such that:
    $$
        \lambda_{\delta_n,p_n}
        \int_{\Omega}
        u_n^2
        \phi_n^2
        dx
    \leq
    C^2
    \delta_n^2
    |\{\phi_n>0\}|
    \leq
        C^3 \delta_n^4
    \leq
        M\delta_n^4;
    $$
    $$
    ||u_n\nabla \phi_n||_{L^2(\Omega)}
    \leq
    |\{\phi_n>0\}|^\frac{1}{2}\cdot
    ||u_n||_{L^\infty(\{\phi_n>0\})}
    ||\nabla \phi_n||_{L^\infty(\{\phi_n>0\})}
    \leq
    C^{\frac{5}{2}}
    \delta_n
    \leq
    M\delta_n.
    $$
    Thus by setting $x_n = ||\phi_n \nabla u_n||_{L^2(\Omega)}$ and substituting in \eqref{randEquationWithHolderAux} the previous inequalities, we obtain:
    $$
    x_n^2
    \leq
        M\delta_n^4
    +
        M\delta_n x_n.
    $$
    From this inequality, the fact that $\delta_n \rightarrow 0$ and Lemma \ref{inequalityLemmaForXEq}, there exists $N>0$ big enough such that, for all $n>N$:
    $$
        \int_{\delta_n K'}
        |\nabla u_n|^2
    \leq
        x_n^2
    \leq
        (2M\delta_n)^2.
    $$
    Now we have that:
    $$
        \int_{K'}
        |\nabla \overline{u}_n|^2
    =
        \frac{1}{\delta_n^2}
        \int_{\delta_n K'}
        |\nabla u_n|^2
        dx
    \leq
        \frac{(2M\delta_n)^2}{\delta_n^2}
    \leq
        (2M)^2,
    $$
    proving the uniform bound.

    This implies that there exists $v \in H^1_{loc}(K)$ and a subsequence $\overline{u}_n$, (which we leave with the same label), such that:
    $$
        \overline{u}_n
    \rightharpoonup_{H^1_{loc}(K)}
        v.
    $$
    Since $\overline{u}_n \geq 0$ this implies $v \geq 0$. Also since $\overline{u}_n \in H^1_0(\frac{1}{\delta_n}(\Omega - B_{\delta_n}(p_n)))$ this implies that:
    $$
        v|_{\partial K} = 0.
    $$
    To prove the harmonicity of $v$, let $\psi \in C_c^\infty(K)$. For $n$ big enough, we have that $supp(\psi) \subset \frac{1}{\delta_n}(\Omega - B_{\delta_n}(p_n))$. Thus:
    \begin{align*}
        \int_{K}
        \langle
            \nabla v,
            \nabla \psi
        \rangle
        dx
    &=
        \lim_n
        \int
        \langle
            \nabla
            \overline{u}_n
            ,
            \nabla \psi
        \rangle
        dx
    =
        -
        \lim_n
        \int
        \Delta \overline{u}_n
        \psi
        dx\\
    &=
        -
        \lim_n
        \delta_n\int
        \Delta u_n(\delta_n x+p_n)
        \psi(x)
        dx
    =
        \lim_n
        \delta_n\lambda_{\delta_n,p_n}\int
        u_n(\delta_n x+ p_n)
        \psi(x)
        dx
    \rightarrow 0,
    \end{align*}
    since $\delta_n \rightarrow 0$, $||u_n||_{L^\infty}$ and $\lambda_{\delta_n,p_n}$ is uniformly bounded by Lemma \ref{unifBoundLemma} and \ref{boundOnEigenValues} and $\psi$ has compact support.

    Finally we prove that $v$ is non-trivial. By Lemma \ref{NonTrivialityLemmaBlowUp}, there exists $\sigma>0$, such that:
    $$
        \int_{\overline{Q}_{\delta_n}(p_n)}
        u_n^2
    \geq
        \sigma \delta_n^3.
    $$
    This implies that:
    $$
        \int_{\frac{1}{\delta_n}\overline{Q}_{\delta_n}(p_n)}
        \overline{u}_n^2
    \geq
        \sigma.
    $$
    We have that the sets:
    $$
        \frac{1}{\delta_n}\overline{Q}_{\delta}(p_n)
    $$
   are exactly equal to:
    $$
        [-1,1] \times [\frac{3}{2},2] \subset int(K).
    $$
    Thus by the weak convergence in $H^1([-1,1] \times [\frac{3}{2},2])$, we have that $u_n$ converges strongly in $L^2([-1,1] \times [\frac{3}{2},2])$ we conclude:
    $$
        \int_{[-1,1] \times [\frac{3}{2},2]}v^2 = \lim_n \int_{[-1,1] \times [\frac{3}{2},2]}\overline{u}_n^2 
        \geq \sigma >0.
    $$
\end{proof}

\begin{lemma}
\label{linearBoundOverHarmonicLimit}
    There exists $C>0$ such that the function $v$ from Lemma \ref{harmonicFunction} satisfies
    \begin{equation}
        |v(x,y)|
    \leq
        C(y+1),
        \quad
        v(x,y)
    =
        v(-x,y).
    \end{equation}
    for all $(x,y)\in K$
\end{lemma}
\begin{proof}
    First we prove the inequality.
    Let $(x,y) \in K$ be an interior point. For $n$ big enough we have that $(x,y) \in \frac{1}{\delta_n}\Omega$. The tangent space at $z(p_n)$ is $T_{z(p_n)}\partial \Omega = \{(x,y): y = -|z(p_n)-p_n|\}$.
    
    Since the boundary $\Omega$ is $C^2$, there exists $\tilde{C}>0$ (independent of $(x,y)$), and a point $(\delta_n x,y_n) \in \partial \Omega$ such that $\dist((\delta_n x,y_n), T_{z(p_n)}\partial \Omega)<\tilde{C}|x|^2\delta_n^2$, that is $|y_n + |z(p_n)-p_n|| \leq \tilde{C}|x|^2\delta_n^2$. We also notice that $|z(p_n)-p_n| = \delta_n + O(\delta_n^2)$.
    
    We can use this boundary point and use Lemma \ref{LipBound} to conclude
    \begin{align}
        \overline{u}_n(x,y)
    &=
        \frac{u(\delta_n(x,y))}{\delta_n}
    \leq
        \frac{C}{\delta_n}
        |\delta_n(x,y)-(\delta_n x,y_n)|
    \leq
        \frac{C|\delta_n y + |z(p_n-p_n)| -(y_n+|z(p_n)-p_n|)| }{\delta_n}\\
    &\leq
        C|y+1| + CO(\delta_n) + C\tilde{C}|x|^2\delta_n.
    \end{align}
    and so taking the limit we conclude that
    \begin{equation}
        v(x,y)
    \leq
        C|y+1|.
    \end{equation}
    
    We now prove the symmetry property. Given a radius $R$ consider the sets given by:
    $$
        V_{1,R}
    :=
        K \cap B_R(0,-1) \cap \{x>0\}
    $$
    $$
        V_{2,R}
    :=
        K\cap B_R(0,-1) \cap \{x<0\}
    $$
    We now consider the function $f: V_{1,R} \rightarrow \mathbb{R}$ given by:
    $$
        f(x,y)
    :=
        v(x,y)
        -
        v(-x,y).
    $$
    Notice that by the definition of this function we have:
    $$
        f|_{\partial V_{1,R}-\partial B_R(0,-1)}
    =
        0.
    $$
    The function is harmonic in $V_{1,R}$ and by Lemma \ref{linearBoundOverHarmonicLimit}, there exists $C>0$ such that:
    \begin{equation}
    \label{ineqForF}
        |f(x,y)| \leq 2C|(x,y)-(0,-1)| \leq 2CR
    \end{equation}
    for $(x,y) \in \partial B_R(0,-1)\cap K$. Now define the sets given by
    $$
        F
    :=
        \left\{
            (x,y) \in \mathbb{R}^2
        :
            -\frac{1}{2}
            \leq
            \frac
            {
                \langle 
                (x,y+1)
                ,
                e_x
                \rangle
            }
            {|(x,y+1)|};
            \quad
            -\frac{1}{2}
            \leq
            \frac{
            \langle 
            (x,y+1)
            ,
            e_y
            \rangle}{|(x,y+1)|}
        \right\}
    $$
    $$
        F'
    :=
        \{
            (x,y)\in\mathbb{R}^2:
            x\geq 0;\quad
            y\geq -1
        \} \subset F.
    $$
    We have that $F$ is a cone at the point $(0,-1)$ with an angle $\alpha = \frac{\pi}{2} + 2\frac{\pi}{6} < \pi$ since $\cos(\frac{\pi}{6}) = \frac{1}{2}$. Let $h : F \rightarrow \mathbb{R}$ be the positive harmonic function given by:
    $$
        h(r,\theta)
    =
        r^\gamma g(\theta),
    $$
    where:
    $$
    g \geq 0;
    \quad
    g\left(-\frac{\pi}{6}\right)= g\left(\frac{\pi}{2}+\frac{\pi}{6}\right) = 0.
    $$
    and $r$ and $\theta$ are the radius and angle measured from the point $(0,-1)$. Also since $\alpha < \pi$ we have that $\gamma > 1$. We have that $c = \min_{\theta \in [0,\frac{\pi}{2}]}g(\theta) > 0$.
    Now define the function given by:
    $$
    \overline{h}_R(r,\theta)
    :=
    \frac{2C}{R^{\gamma-1} c}
    h(r,\theta),
    $$
    which satisfies
    $$
    \overline{h}_R(r,\theta)|_{F' \cap \partial B_R(-1,0)}
    =
    \frac
    {2C R^\gamma g(\theta)}
    {R^{\gamma - 1} c}
    \geq
    2CR
    \geq
    |f(x,y)|.
    $$
    Thus we conclude that $-\overline{h}_R|_{\partial V_{1,R}}\leq f|_{\partial V_{1,R}} \leq \overline{h}_R|_{\partial V_{1,R}}$. Since both of these are harmonic, by maximum principle we have that:
    $$
      - \overline{h}_{R}(z) \leq f(z) \leq \overline{h}_R(z).
    $$
    for all $z \in V_{1,R}$. Since $\gamma > 1$, we have that $\overline{h}_{R} \rightarrow 0$ uniformly on compact sets as $R \rightarrow \infty$, and thus $f = 0$ proving that $v(x,y) = v(-x,y)$. 
\end{proof}

\begin{lemma}
\label{gradientControl}
    Let $w = (\frac{3}{2},0)$. There exists a constant $C>0$ independent of $n$ such that:
    $$
        \sup_{n}
        |\nabla \overline{u}_n(w)|\leq C.
    $$
\end{lemma}
\begin{proof}
    We start by noticing that for $n$ big enough, the ball $B_{\frac{1}{4}}(w)$ is contained in the domain of $\overline{u}_n$.
    Now define the function given by:
    $$
        v_n(z):= \overline{u}_n(z) - \overline{u}_n(w) - \langle \nabla \overline{u}_n(w), z -w \rangle.
    $$
    We have that:
    $$
        \Delta v_n(z)
    =
        -\delta_n^2 \lambda_{\delta_n,p_n}\overline{u}_n(z)
        \leq 0.
    $$
    Thus if we define:
    $$
        \varphi_n(r)
    :=
        \frac{1}{|\partial B_r|}
        \int_{\partial B_r(w)}
        v_n
        d\mathcal{H}^1
    $$
    then:
    $$
        \varphi_n'(r)
    :=
        \frac{1}{|\partial B_r|}
        \int_{B_r(w)}
        \Delta v_n
        d\mathcal{H}^1
        \leq 0.
    $$
    Since $v_n(0) = 0$ this implies that:
    \begin{equation}
    \label{AuxStuffBound1}
        \int_{\partial B_r(w)}
        v_n
        d\mathcal{H}^1 \leq 0
    \end{equation}
    for all $r \in ]0,\frac{1}{4}[$. After a rotation, we assume without loss of generality that $\nabla \overline{u}_n(w) = -\alpha_n(1, 0)$. We then have that:
    $$
        \frac{\partial v_n}        
        {\partial x}(w)
    =
        0.
    $$
    Also we have that:
    $$
        -\Delta \frac{\partial v_n}{\partial x}(z)
    =
        \delta_n^2
        \lambda_{\delta_n,p_n}
        \frac
        {\partial \overline{u}_n}
        {\partial x}(z)
    $$
    By Lemma \ref{harmonicFunction}, since $B_{\frac{1}{4}}(w)\subset int(K)$, there exists $C>0$ such that:
    $$
        ||\nabla \overline{u}_n||_{L^2(B_{\frac{1}{4}}(w))}
        \leq
        C
    $$
    From Holder's inequality we then conclude that for $r \in ]0,\frac{1}{4}[$:
    \begin{equation}
    \label{L1BoundForDerivative}
        ||\nabla \overline{u}_n||_{L^1(B_r(w))}
    \leq
        ||\chi_{B_r(w)}||_{L^2(B_\frac{1}{4}(w))}
        ||\nabla \overline{u}_n||_{L^2(B_{\frac{1}{4}}(w))}
    \leq
        Cr.
    \end{equation}
    Now define:
    $$
        \phi_n(r)
    :=
        \frac{1}{|\partial B_r|}
        \int_{\partial B_r(w)}
        \frac{\partial v_n}{\partial x}
        d\mathcal{H}^1.
    $$
    Using \eqref{L1BoundForDerivative} and Lemma \ref{boundOnEigenValues}, we conclude:
    $$
        |\phi_n'(r)|
    =
    \left|
        \frac{1}{|\partial B_r|}
        \int_{B_r(w)}
        \Delta \frac{\partial v_n}{\partial x}
    \right|
    \leq
        \frac{1}{|\partial B_r|}
        \int_{B_r(w)}
        \delta_n^2
        \lambda_{\delta_n,p_n}
        \left|
        \frac
        {\partial \overline{u}_n}{\partial x}
        \right|
    \leq
        \delta_n^2\frac{Cr}{r|\partial B_1|}
    \leq
        {C}|\partial B_1|\delta_n^2,
    $$
    replacing $C$ by $C|\partial B_1|^2$ if necessary.
    Since $\nabla v_n(w) = 0$, integrating the previous equation, we have for all $r \in ]0,\frac{1}{4}[$:
    $$
        \frac{1}{|\partial B_r|}
        \int_{\partial B_r(w)}
        \frac{\partial v_n}{\partial x}
        d\mathcal{H}^1
    =
        \phi_n(r)
    \leq
        C\delta_n^2|\partial B_1|r = C\delta_n^2 |\partial B_r|.
    $$
    Thus we obtain:
    \begin{align*}
        \int_{B_r(w)}
        \frac{\partial v_n}
        {\partial x}
    &=
        \int_{0}^r
        \left(
        \int_{\partial B_s(w)}
        \frac{\partial v_n}
        {\partial x}
        \mathcal{H}^1
        \right)
        ds
    \leq
        \int_{0}^r
        \left(
        \frac{|\partial B_s|}{|\partial B_s|}
            \int_{B_s(w)}
        \frac{\partial v_n}
        {\partial x}
        \mathcal{H}^1
        \right)
        ds\\
    &=
        \int_{0}^r
        \left(
            C\delta_n^2
            |\partial B_s|^2
        \right)
        ds
    =
        \frac{C\delta_n^2|\partial B_1|^2}
        {3}r^3.
    \end{align*}
    Using integration by parts we also have that:
    \begin{equation}
    \label{derivEqForAuxStuff}
        \int_{B_r(w)}
        \frac{\partial v_n}
        {\partial x}
    =
        \int_{\partial B_r(w)}
        v_n(z)(\nu_x)_z
        d\mathcal{H}^1
    \leq
        \frac{C\delta_n^2|\partial B_1|^2}
        {3}r^3.
    \end{equation}
    Summing equations \eqref{AuxStuffBound1} and \eqref{derivEqForAuxStuff} we obtain:
    \begin{align*}
    \int_{\partial B_r(w)}
        v_n(z)
        (1+(\nu_x)_z)d\mathcal{H}^1
    &\leq
        \frac{C\delta_n^2|\partial B_1|^2}
        {3}r^3,
    \end{align*}
    where $(\nu_x)_z = \frac{(z-w)_x}{|z-w|}$, is the $x$ coordinate of the outer normal $\nu$ of $\partial B_{|z-w|}(w)$ at $z$. Also since $\overline{u}_n$ are positive in $B_{\frac{1}{2}}(w)$ we have:
    $$
        v_n(z)
        \geq
        -
        \overline{u}_n(w)
        -\langle \nabla \overline{u}_n, z-w \rangle
    =
        -
        \overline{u}_n(w)
        +
        \alpha_n
        (z-w)_x.
    $$
    Defining $a_n = \overline{u}_n(w)$, we obtain:
    \begin{align*}
        \frac{\delta_n^2|\partial B_1|^2}
        {3}r^3
        &\geq
        \int_{\partial B_r(w)}
        \left(-a_n  + \alpha_n(z-w)_x\right)
        \left(1+\frac{(z-w)_x}{|z-w|}
        \right)d\mathcal{H}^1\\
        &=
        \int_{\partial B_r(w)}
        -a_n + \alpha_n\frac{(z-w)_x^2}{|z-w|}
        d\mathcal{H}^1
        =
        -a_n|\partial B_r|
        +\alpha_n
        \frac{r|\partial B_r|}{2},
    \end{align*}
    where we have used symmetry of $(z-w)_x$ in $\partial B_r(w)$.
    Since this is true for all $n \in \mathbb{R}$, since $\delta_n \rightarrow 0$ and $a_n = \overline{u}_n(w)$ is uniformly bounded by Lemma \ref{harmonicFunction}, if $\alpha_n = |\nabla \overline{u}_n(w)| \rightarrow \infty$ the inequality above cannot hold, thus there exists $C>0$ such that for all $n \in \mathbb{N}$ we have:
    $$
        |\nabla \overline{u}_n(w)| \leq C,
    $$
    concluding the proof.
\end{proof}

We can apply Theorem \ref{boundaryReg} to conclude the following

\begin{lemma}
\label{ConnectingConvergence}
    Let $v$ be the function from Lemma \ref{harmonicFunction} and define the subset of $K$:
    $$
        V
    :=
         \{
            (x,y) \in B_2(0) - B_1(0):
            y \geq -\cos(\theta_0)
        \}.
    $$
    Then:
    $$
        ||\nabla \overline{u}_n - \nabla v||_{L^\infty(V)}
    \rightarrow
        0.
    $$
\end{lemma}
\begin{proof}
    
    Using theorem \ref{boundaryReg} and a change of variables, since:
    $$
        -\Delta \overline{u}_n
    =
        \lambda_{\delta_n,p_n}
        \delta_n
        \overline{u}_n
    $$
    we can conclude that there exists $C>0$ such that in the set $V$ we have:
    \begin{equation}
    \label{2ndDerivBound}
        ||D^2 \overline{u}_n||_{L^\infty(V)}
    \leq
        C
        \left(
            1
        +
            ||\overline{u}_n||_{L^\infty(B_2(0) - B_1(0))}
        +
            ||\lambda_{\delta_n,p_n}\delta_n \overline{u}_n
            ||_{L^\infty(B_2(0) - B_1(0))}
        \right).
    \end{equation}

    Also, by Lemma \ref{harmonicFunction} we can conclude that there exists $C>0$ such that $||\overline{u}_n||_{L^\infty(B_2(0)- B_1(0))} \leq 2C$. Thus $||D^2 \overline{u}_n||_{L^\infty(V)}$ is uniformly bounded. Now given the point $w = (\frac{3}{2}, 0) \in V$, by Lemma \ref{gradientControl} we can conclude that there exists $C>0$ such that:
    \begin{equation}
    \label{gradPointBound}
        |\nabla \overline{u}_n (w)| 
    \leq 
        C.
    \end{equation}
    for all $n$. Combining \eqref{2ndDerivBound} and \eqref{gradPointBound} we can conclude that there exists $C>0$ such that:
    \begin{equation}
    \label{unifBoundOfGrad}
        ||\nabla \overline{u}_n||_{L^\infty(V)}
    \leq
        C.
    \end{equation}
    From \eqref{2ndDerivBound}, by Sobolev inequalities, we know that $\overline{u}_n$ are uniformly Lipschitz in $V$, and by \eqref{unifBoundOfGrad} we know that they are also uniformly bounded in $V$. Thus using Ascoli-Arzelà's Theorem, there exists a subsequence $\nabla \overline{u}_n$ (which we leave with the same label) such that it converges uniformly in $V$ (and by unicity of limits it must be $\nabla v$). From the convergence above and by Lemma \ref{harmonicFunction} we conclude that:
    $$
        ||
            \nabla \overline{u}_n
        -
            \nabla v
            ||_{L^\infty(V)}
        \rightarrow
        0,
    $$
    finishing the proof.
\end{proof}

Now we construct a blowdown sequence from the function $v$. Let $R_n>2$ be a sequence such that $R_n \rightarrow + \infty$.

We will construct the following blowdown sequence of the harmonic function $v$, given by:
$$
    \tilde{v}_n(x,y)
:=
    \frac{1}{R_n}
    v(R_n(x,y) + (0,-1)).
$$

\begin{lemma}
\label{linearConvergence}
    There exist $\tilde{v} : \{(x,y): y \geq 0\} \rightarrow \mathbb{R}$ such that:
    $$
        ||\nabla \tilde{v}_n- \nabla \tilde{v}||_{L^\infty(B_2(0) -(B_1(0) \cup \{ y>0\}))}
    \rightarrow 0,
    $$
    and some $\alpha \geq 0$ such that:
    $$
        \tilde{v}(x,y)
    =
        \alpha y.
    $$
\end{lemma}
\begin{proof}
    Let $D := B_3(0) - \left(B_1(0) \cup \{y>0\}\right)$.

\begin{figure}
    \centering
\tikzset{every picture/.style={line width=0.75pt}} 

\begin{tikzpicture}[x=0.75pt,y=0.75pt,yscale=-1,xscale=1]
\draw   (320.06,209.83) .. controls (320.06,203.05) and (325.55,197.56) .. (332.33,197.56) .. controls (339.11,197.56) and (344.61,203.05) .. (344.61,209.83) .. controls (344.61,216.61) and (339.11,222.11) .. (332.33,222.11) .. controls (325.55,222.11) and (320.06,216.61) .. (320.06,209.83) -- cycle ;
\draw  [fill={rgb, 255:red, 210; green, 210; blue, 210 }  ,fill opacity=1 ] (118.15,222.26) .. controls (118.15,222.21) and (118.15,222.16) .. (118.15,222.11) .. controls (118.15,103.82) and (214.04,7.93) .. (332.33,7.93) .. controls (450.62,7.93) and (546.52,103.82) .. (546.52,222.11) -- (410.01,222.11) .. controls (410.01,179.21) and (375.23,144.44) .. (332.33,144.44) .. controls (289.44,144.44) and (254.66,179.21) .. (254.66,222.11) .. controls (254.66,222.13) and (254.66,222.15) .. (254.66,222.17) -- cycle ;
\draw  [fill={rgb, 255:red, 170; green, 170; blue, 170 }  ,fill opacity=1 ] (174.52,221.76) .. controls (174.95,134.87) and (245.52,64.56) .. (332.51,64.56) .. controls (419.78,64.56) and (490.52,135.3) .. (490.52,222.56) .. controls (490.52,222.59) and (490.52,222.61) .. (490.52,222.64) -- (410.37,222.6) .. controls (410.37,222.59) and (410.37,222.57) .. (410.37,222.56) .. controls (410.37,179.56) and (375.51,144.7) .. (332.51,144.7) .. controls (289.65,144.7) and (254.87,179.35) .. (254.66,222.17) -- cycle ;

\draw    (54,222.11) -- (610.67,222.11) ;


\draw   (217.73,221.98) .. controls (217.31,217.99) and (217.1,213.94) .. (217.1,209.83) .. controls (217.1,146.19) and (268.69,94.6) .. (332.33,94.6) .. controls (395.98,94.6) and (447.57,146.19) .. (447.57,209.83) .. controls (447.57,213.67) and (447.38,217.46) .. (447.02,221.2) ;  

\draw (329.33,227.51) node [anchor=north west][inner sep=0.75pt]  [font=\footnotesize]  {$B_{\frac{1}{R_{n}}}\left( 0,\frac{1}{R_{n}}\right)$};
\draw (232,55.07) node [anchor=north west][inner sep=0.75pt]    {$D$};
\draw (303,73.9) node [anchor=north west][inner sep=0.75pt]  [font=\footnotesize]  {$D\cap B_{2}( 0)$};
\draw (268.17,111.23) node [anchor=north west][inner sep=0.75pt]  [font=\scriptsize]  {$\partial B_{\frac{3}{2}}\left( 0,\frac{1}{R_{n}}\right) \cap \{y >0\}$};

\end{tikzpicture}
\caption{}
\label{fig:enter-label}
\end{figure}
    
    By Theorem \ref{boundaryReg}, since $v$ is harmonic, we can conclude that there exists $C>0$ such that:
    \begin{equation}
        \label{2ndDivBoundOfBlowDown}
        ||D^2{\tilde{v}}_n||_{L^\infty(B_2\cap D)}
    \leq
        C
        \left(
            1
        +
            ||\tilde{v}_n||_{L^\infty(D)}
        \right).
    \end{equation}
    Also using Lemma \ref{harmonicFunction} we can conclude that there exists $C>0$ such that $||\tilde{v}_n||_{L^\infty(D)} \leq 2C$, and so $||D^2 \tilde{v}_n||_{L^\infty(B_2\cap D)}$ is uniformly bounded. Also given the point $w = (\sqrt{3}/2, \sqrt{3}/2) \in B_2^+(0) - B_1(0)$, since $\Tilde{v}_n$ are positive and harmonic, we can do a proof similar to Lemma \ref{gradientControl}, to conclude that there exists $C>0$ such that:
    \begin{equation}
    \label{unifBoundGradBlowdown}
        \sup_n |\nabla \tilde{v}_n(w)|
    \leq
        C.
    \end{equation}
    From equation \eqref{unifBoundGradBlowdown} and the uniform bound of $||D^2 \Tilde{v}_n||_{L^\infty(B_2^+ - B_1(0))}$ we must have that $\nabla v_n$ must converge in the uniform topology in $B_2^+(0) - B_1(0)$. Generalizing this argument, substituting $B_1(0)$ by $B_\epsilon(0)$ a smaller ball, we can also conclude that:
    $$
        \tilde{v}_n \rightarrow_{H^1_{loc}(\mathbb{H})} \tilde v
    $$
    where $\tilde{v}$ is harmonic in the upper half plane $\mathbb{H}$ and $\tilde{v}|_{\partial \mathbb{H}} = 0$. From Lemma \ref{linearBoundOverHarmonicLimit} we have that $\tilde{v}_n(x,y) \leq Cy$. Thus, since $\tilde{v}$ is positive and it satisfies a bound $\tilde{v}(x,y) \leq C y$, from the fact it is harmonic, we can conclude that:
    $$
        \tilde{v}(x,y) = \alpha y
    $$
    for some $\alpha \geq 0$. From the uniform convergence of the gradients the proof follows.
\end{proof}

\begin{lemma}
\label{positiveAngleHarmonic}
    Let $v$ be the function from Lemma \ref{harmonicFunction}.
    We must have that:
    $$
        \frac
        {\partial v}
        {\partial \theta}
        (x,y)
        \geq
        0
    $$
    for $x> 0$ and $(x,y) \in K$ defined in equation \eqref{upHalfPlaneWithoutBall}.
\end{lemma}
\begin{proof}
    By Lemma \ref{linearBoundOverHarmonicLimit}, we must have:
    $$
        \frac{\partial v}
        {\partial \theta}
        (0,y)
    =
        0.
    $$
    Also since $v$ is positive, and $v(x,-1) = 0$, for $x>0$ we must have:
    $$
        \frac{\partial v}
        {\partial \theta}(x,-1)
    \geq
        0.
    $$
    Also for $(x,y) \in \partial B_1(0)$ we have:
    $$
        \frac{\partial v}
        {\partial \theta}(x,y)
    =
        0.
    $$
    Given the sequence $R_n \rightarrow \infty$, define:
    $$
        D_n
    :=
        \{x\geq 0\}
        \cap 
        \{y\geq -1\}
        \cap
        B_{\frac{3}{2}R_n}(0)
        -
        B_1(0).
    $$
    If we prove that:
    $$
        \lim_n\inf_{(x,y) \in \partial B_{\frac{3}{2}R_n}(0) \cap \{y\geq -1; x\geq 0\}}
            \frac{\partial v}
            {\partial \theta}
        \geq
            0
    $$
    since $\frac{\partial v}{\partial \theta}$ is harmonic we conclude that:
    \begin{equation}
    \label{positivityOfAngularDeriv}
        \inf_{(x,y) \in D_n}
        \frac{\partial v}{\partial \theta}(x,y)
        \geq
        0,
    \end{equation}
    for all $n$. From Lemma \ref{linearConvergence} there exists $\alpha \geq 0$ such that we have that:
    $$
    ||
            \nabla v
            -
            \alpha e_y
        ||_{L^\infty(\partial B_{\frac{3}{2}R_n}(0)\cap \{y> -1\})}
    =
        ||\nabla \tilde{v}_n - \alpha e_y||_{L^\infty(\partial B_{\frac{3}{2}}(0,\frac{1}{R_n}) \cap\{y>0\})}
    \rightarrow
        0.
    $$
    Thus in particular:
    $$
        \bigg|\bigg|\frac{\partial v}{\partial \theta} - \alpha \langle e_y, \frac{\partial}{\partial \theta}_{(x,y)} \rangle\bigg|\bigg|_{L^\infty(\partial B_{\frac{3}{2}R_n}(0) \cap \{y>-1\})} \longrightarrow 0.
    $$
    Since $\langle e_y, \frac{\partial}{\partial \theta}_{(x,y)} \rangle \geq 0$ for points $(x,y)$ with $x > 0$, we conclude:
    $$
        \lim_n
        \inf_{(x,y) \in \partial B_{R_n}(0) \cap \{y\geq -1; x\geq 0\}}
        \frac
        {\partial v}
        {\partial \theta}(x,y)
    \geq
        0.
    $$
    Thus we conclude equation \eqref{positivityOfAngularDeriv} is satisfied for all $n$.
    and because $R_n \rightarrow \infty$, this concludes the proof.
\end{proof}

\begin{lemma}
\label{positiveDeriv}
    There exists $\gamma > 0$ such that for $$z \in {A}_1'
    := \{(x,y) \in \partial B_1(0) : -\cos(\theta_0) \leq \langle (x,y) , e_y \rangle \leq \cos(\theta_0); x \geq 0\}$$ we have:
    $$
        \frac
        {\partial^2 v}
        {\partial \theta \partial \nu}
        (z)
    \geq
        \gamma.
    $$
\end{lemma}
\begin{proof}
    Using H\"opf's Lemma, Lemma \ref{positiveAngleHarmonic} and the fact that $v$ is non-trivial (which implies $\frac{\partial v}{\partial \theta}$ is non-trivial), we conclude that for $(x,y) \in \partial B_1(0)$ with $x > 0$ we must have:
    $$
        \frac
        {\partial^2 v}
        {\partial \nu \partial \theta}
        (x,y)
    =
        \frac
        {\partial^2 v}
        {\partial \theta \partial \nu}
        (x,y)
        >
        0.
    $$
    Thus using $C^2$ continuity in the set $A_1'$, we conclude that there exists $\gamma > 0$ such that for all $(x,y) \in \{(x,y) \in \partial B_1(0) : -\cos(\theta_0) \leq \langle (x,y) , e_y \rangle \leq \cos(\theta_0); x \geq 0\}$ we have:
    $$
         \frac
        {\partial^2 v}
        {\partial \theta \partial \nu}
        (x,y)
        \geq\gamma.
    $$
\end{proof}

Now we conclude the proof Proposition \ref{positiveIntegrals}.
\begin{proof}
    By Lemma \ref{positiveDeriv} we conclude that in the set $A_1'$, there exists $c>0$ such that:
    $$
        \int_{{A}_1'}
            \frac{\partial v}
                {\partial \nu}
                \frac{\partial v}
                {\partial y}
                d\mathcal{H}^1
        =
            \int_{{A}_1'}
                \left(
                    \frac{\partial v}
                    {\partial \nu}
                \right)^2
                \langle e_y, \nu \rangle
                d\mathcal{H}^1
            \geq
        c>0.
    $$
    This is the case since $(\frac{\partial v}{\partial \nu})^2$ increases as $\theta$ increases and $\langle e_y, \nu \rangle$ is an odd function in $\theta$.
    
    By the convergence from Lemma \ref{ConnectingConvergence}, we conclude that for $n \geq N$ where $N$ is big enough we have:
    \begin{equation}
        \int_{{A}_1'}
         \frac{\partial \overline{u}_n}
        {\partial \nu}
                \frac{\partial \overline{u}_n}
                {\partial y}
                d\mathcal{H}^1
            >
            0
    \end{equation}
    but this implies that
    $$
        \int_{A_1}
         \frac{\partial u_n}
        {\partial \nu}
                \frac{\partial u_n}
                {\partial y}
                d\mathcal{H}^1
            >
            0
    $$
    in contradiction with the initial hypothesis \eqref{contradictHyp}. This concludes the proof reaching contradiction.
\end{proof}

\subsection{Top Integral Lower Bound}
    In this section we prove the following proposition:
    \begin{prop}
        \label{topInt}
            If $\theta \in ]0,\frac{\pi}{2}[$, there exists $\Tilde{\delta} > 0$ and $c>0$ depending only on $\Omega$ such that if $\delta < \Tilde{\delta}$ and $\dist_{min}(\partial B_\delta(p), \partial \Omega) < \delta^2$ then:
            $$
                \int_{C_{\theta,\delta}^+(p)}
                \left(
                \frac
                 {\partial u_{\delta,p}}
            {\partial \nu}
            \right)
            \left(
            \frac
            {\partial u_{\delta,p}}
            {\partial y}
            \right)
            d\mathcal{H}^1
            \geq
            c\theta \delta
            $$
            where:
            $$
            C_{\theta,\delta}^+(p) =
                \left\{z \in \partial B_\delta(p):
    \frac{\langle z-p, e_y \rangle}{|z-p|}
    \geq
    \cos(\theta)
                \right\}.
            $$
            and $\nu$ is the exterior normal to $B_\delta(p)$.
        \end{prop}
\begin{proof}
        The proof will be done by contradiction as in the section for the side integrals. Thus we suppose by contradiction that there exists a sequence $\theta_n \in ]0, \frac{\pi}{2}[$, $\delta_n \rightarrow 0$ and $p_n$ such that $\dist_{min}(\partial B_{\delta_n}(p_n),\partial \Omega)$ satisfying:
        \begin{equation}
        \label{convergenceTozeroContradictionStuff}
            \frac{1}{\theta_n\delta_n}
            \int_{C_{\theta_n,\delta_n}^+(p_n)}
            \left(
                \frac
                 {\partial u_n}
            {\partial \nu}
            \right)
            \left(
            \frac
            {\partial u_n}
            {\partial y}
            \right)
            d\mathcal{H}^1
            \rightarrow
            0.
        \end{equation}
        By considering the blowup given by:
        $$
            \overline{u}_n(x)
:=
    \frac
    {u_n}
    {\delta_n}(\delta_n x + p_n),
        $$
        we can use Lemmas \ref{harmonicFunction} and \ref{ConnectingConvergence} to conclude there exists $v: K \rightarrow \mathbb{R}$ harmonic in $K$ such that:
        \begin{equation}
        \label{convergenceOfDerivativesInCircleTop}
            ||\nabla v - \nabla \overline{u}_n||_{L^\infty(B_2^+(0) - B_1(0))}
            \rightarrow
            0
        \end{equation}
        where $K = \{(x,y):y\geq -1\} - B_1(0)$. By \ref{harmonicFunction} $v$, is a non-trivial positive harmonic function, thus by H\"opf's Lemma, we conclude there exists $c>0$ such that:
        \begin{equation}
        \label{HopfsLemmaStuff}
            \sup_{z \in B_1^+(0)}|\nabla v(z)|
            \geq
            c>0
        \end{equation}
        Using \eqref{convergenceOfDerivativesInCircleTop} and \eqref{HopfsLemmaStuff}, for $n$ big enough we have that:
        $$
            \sup_{z \in C_{\theta_n,\delta_n}^+(p_n)}
            \frac{\partial \overline{u}_n}
            {\partial \nu}(z)
        \geq
            \frac{c}{2}.
        $$
        This implies that for $n$ big enough there exists $\alpha >0$ such that:
        \begin{align}
            \frac{1}{\theta_n\delta_n}
            \int_{C_{\theta_n,\delta_n}^+(p_n)}
            \left(
                \frac
                 {\partial u_n}
            {\partial \nu}
            \right)
            \left(
            \frac
            {\partial u_n}
            {\partial y}
            \right)d\mathcal{H}^1
        &=
            \frac{1}{\theta_n}
            \int_{C_{\theta_n,1}(0)}
            \left(
                \frac
                 {\partial \overline{u}_n}
            {\partial \nu}
            \right)
            \left(
            \frac
            {\partial \overline{u}_n}
            {\partial y}
            \right)d\mathcal{H}^1\\
        &\geq
            \frac{(c/2)^2}{\theta_n}
            \int_{C_{\theta_n,1}(0)}
            \langle e_y, \nu_z \rangle
            d\mathcal{H}^1
        \geq \alpha>0.
    \label{lowBoundContradiction}
        \end{align}
        We have that \eqref{lowBoundContradiction} contradicts the hypothesis \eqref{convergenceTozeroContradictionStuff}, which concludes the proof.
    \end{proof}

\section{Conclusion of Theorem}
Now we conclude the proof of Theorem \ref{RealMainTheorem}.

\begin{proof}
    Let $c>0$ be the constant from Proposition \ref{topInt} and $C>0$ the constant from Proposition \ref{bottomThm}. Choose $\theta >0$ such that:
    $$
        C\theta^2 < c\theta.
    $$
    By Propositions \ref{topInt} and \ref{bottomThm} choose $\Tilde{\delta} = \Tilde{\delta}(\Omega, \theta)$ small enough such that for $\delta < \Tilde{\delta}$ and $\dist_{min}(\partial B_\delta(p),\partial \Omega) <\delta^2$ we have:
    \begin{equation}
    \label{eq1}
        \int_{C_\theta^+}
        \frac{\partial u_{\delta,p}}
        {\partial \nu}
                \frac{\partial u_{\delta,p}}
                {\partial y}
                d\mathcal{H}^1
                \geq
                c\delta\theta;
    \end{equation}
    \begin{equation}
    \label{eq2}
        \int_{C_\theta^-}
        \frac{\partial u_{\delta,p}}
        {\partial \nu}
                \frac{\partial u_{\delta,p}}
                {\partial y}
                d\mathcal{H}^1
                \geq
                -C\delta\theta^2.
    \end{equation}
    Also for $A = \partial B_{\delta}(p) - \left(C_\theta^+ \cup C_\theta^-\right)$, by Proposition \ref{positiveIntegrals} there exists $\Tilde{\delta}(\theta, \Omega)$ small enough such that for $\delta < \Tilde{\delta}$ then:
    \begin{equation}
        \label{eq3}
        \int_{A}
        \frac{\partial u_{\delta,p}}
        {\partial \nu}
                \frac{\partial u_{\delta,p}}
                {\partial y}
                d\mathcal{H}^1
                \geq
                0.
    \end{equation}
    Combining equation \eqref{eq1}, \eqref{eq2} and \eqref{eq3} we conclude that:
    $$
        \int_{\partial B_\delta(p)}
        \frac{\partial u_{\delta,p}}
        {\partial \nu}
                \frac{\partial u_{\delta,p}}
                {\partial y}
                d\mathcal{H}^1
        \geq
        \delta(
        c\theta
        -
        C\theta^2
        )
        >
        0,
    $$
    concluding the proof.
\end{proof}
\printbibliography
\end{document}